\documentclass[a4paper,reqno]{amsart}
\pdfoutput=1

\usepackage{enumerate} 
\usepackage{enumitem}

\usepackage{amsmath,amsfonts,amssymb,amsthm,stmaryrd,bbm,graphicx,mathtools,enumerate}
\usepackage{mathrsfs}
\usepackage[T1]{fontenc} 
\usepackage[latin1]{inputenc}
\usepackage[english]{babel}
\usepackage[tracking,spacing,kerning,babel]{microtype}
\usepackage{wasysym} 
\usepackage{color}
\usepackage{xcolor}
    	\usepackage{framed} 
\usepackage{wrapfig} 

\usepackage[final]{hyperref}   
\hypersetup{
    linktoc=page,
    linkcolor=red,          
    citecolor=blue,        
    filecolor=blue,      
    urlcolor=cyan,
   colorlinks=true           
}

\usepackage{lipsum} 

\usepackage{minitoc}
\usepackage{multicol}

\newtheoremstyle{mine}
{\baselineskip}
{\baselineskip}
{\itshape}
{
}
{\bfseries}
{.}
{.5em}
{#1 #2\ifx#3\relax\else~(#3)\fi}

\theoremstyle{mine}

\newtheorem{theorem}{Theorem}
\numberwithin{theorem}{section}
\newtheorem{corollary}[theorem]{Corollary}
\newtheorem{proposition}[theorem]{Proposition}
\newtheorem{lemma}[theorem]{Lemma}

\newtheorem{definition}[theorem]{Definition}
 
\newtheorem{question}{Question}
\numberwithin{equation}{section}

\theoremstyle{remark}
\newtheorem{remark}{Remark}

\colorlet{shadecolor}{blue!10}

\def\rm{\reversemarginpar}

\let\qed=\QED

\renewcommand{\epsilon}{\varepsilon}

\newcommand{\R}{\mathbb{R}}

\newcommand{\Cb}{\mathbf{C}}
\newcommand{\Z}{\mathbb{Z}}
\newcommand{\N}{\mathbb{N}}

\def\calB{\mathcal{B}}

\def\calD{\mathcal{D}}
\def\calE{\mathcal{E}}

\def\calN{\mathcal{N}}

\def\P{\mathbb{P}} 
\def\E{\mathbb{E}} 
\def\md{\mid}

\def\Bb#1#2{{\def\md{\bigm| }#1\bigl[#2\bigr]}}

\def\Eb{\Bb\E}

\def\EFK#1#2#3{{\def\md{\bigm| } \E_{#1}^{\,#2}  \bigl[  #3 \bigr]}}

\def\<#1{\langle #1\rangle}
\newcommand{\red}[1]{{\color{red}#1}}

\newcommand{\purple}[1]{{\color{purple}#1}}

\newcommand{\xinxin}[1]
{\textcolor{blue}{*}\marginpar[\textcolor{blue} {  \raggedleft  \footnotesize \textbf{Xinxin:}  #1 }  ]{ \textcolor{blue} { \raggedright  \footnotesize  \textbf{Xinxin:} #1 }  }}

\def\nn{\nonumber}
\def\bi{\begin{itemize}}  
\def\ei{\end{itemize}}
\def\bnum{\begin{enumerate}} 
\def\enum{\end{enumerate}}
\def\ni{\noindent}
\def\bf{\bfseries}

\def\PPP{\mathrm{PPP}}
\def\M{\mathscr{M}}

\newcommand{\ind}[1]{\mathbf{1}_{\left\{#1\right\}}}
\newcommand{\mS}{\underline{S}}

\DeclareMathOperator*{\esssup}{ess\,sup}

\title
[Fixed points of BBM]
{
The fixed points of Branching Brownian Motion
}

\author{Xinxin Chen, Christophe Garban,  Atul Shekhar}

\address
{Universit\'e Claude Bernard Lyon 1, CNRS UMR 5208, Institut Camille Jordan, 69622 Villeurbanne, France}
\email{chen@math.univ-lyon1.fr; shekhar@math.univ-lyon1.fr}

\address
{Universit\'e Claude Bernard Lyon 1, CNRS UMR 5208, Institut Camille Jordan, 69622 Villeurbanne, France \, and Institut Universitaire de France (IUF)}
\email{garban@math.univ-lyon1.fr}

\begin{document}

\maketitle

\begin{center}
{\em Dedicated to the memory of Tom Liggett}
\end{center}

\begin{abstract}
In this work, we characterize all the point processes $\theta=\sum_{i\in \N} \delta_{x_i}$ on $\R$ which are left invariant under branching Brownian motions with critical drift $-\sqrt{2}$. Our characterization holds under the only assumption that $\theta(\R_+)<\infty$ almost surely. 
\end{abstract}


\section[\;\;\;\;Introduction]{Introduction}\label{intro}

\subsection{Context.} 

Binary branching Brownian motion\footnote{For simplicity, we shall only consider the case of binary branching in this paper, but the main results hold also under the more general setting of \cite{ABK13}.}  can be described as follows: particles evolve independently of each other in $\R$ and split into two independent particles at rate one. If one starts such a BBM process with a single particle at the origin at time 0, it is well known that at time $t$, there will be  $n(t)\approx e^t$ particles whose positions will be denoted by $\{\chi_k(t)\}_{1\leq k \leq n(t)}$. Furthermore the rightmost particle at time $t$, i.e. $M_t:=\sup_{k\leq n(t)} \chi_k(t)$ will be found modulo $O(1)$-fluctuations at distance $m(t)= \sqrt{2} t - \frac{3}{2\sqrt{2}} \log_+(t)$. See for example \cite{bovier-book,shi-book} and references therein.

This stochastic process has attracted a lot of attention in many different contexts. For example it happens to be  strongly  connected with PDEs since McKean's observation \cite{McKean75} that the fluctuations of rightmost particles are described by the solutions of the FKPP equation
\begin{align}\label{}
u_t = \frac 1 2 u_{xx} + u(1-u)\,,\nn
\end{align}
and in particular by the {\em traveling wave} solutions $x\mapsto \omega_\lambda(x)$ of this non-linear PDE (especially the critical one at $\lambda_c=\sqrt{2}$). 
See in particular the works \cite{bramson78,bramson83} and the book \cite{bovier-book}. 
In a different context, BBM also attracted a lot attention in physics via its relationships with the GREM model or the fact that (up to a minus sign) leading particles can be viewed as the lowest energies of a directed polymer in a random medium (in a mean-field regime), see for example \cite{derrida2009,derrida2011,bovier-book}. This model has also natural connections with mathematical biology due to its links with reaction-diffusion models.

As it has been pioneered since the work of Brunet-Derrrida \cite{derrida2009,derrida2011}, it is natural to consider the BBM process viewed from its frontier. This means that one considers the point process $\calE_t$  of the BBM particles shifted by $m(t)$, i.e.
\begin{equation}\label{abk-conv}
    \mathcal{E}_t := \sum_{k=1}^{n(t)}\delta_{\chi_k(t)- m(t)}\,.
\end{equation}
It has been proved independently in the seminal works \cite{ABK13,aidekon13}) that this shifted BBM converges as a random point process to a limiting Point Process $\calE_\infty$ with an interesting structure. Indeed $\calE_\infty$ has the law of an explicit decorated Poisson Point Process. The intensity of its underlying  PPP is rather simple ($e^{-\sqrt{2} x} dx$ up to a random translation coming from the so-called {\em derivative martingale}) but the law of its decoration turns out to more intriguing. See Section \ref{ss.extreme} for more details as well as the reference \cite{cortines2019}.

In this work, we are interested in identifying all the {\em fixed points of BBM.} Clearly one cannot hope to find any fixed point if one considers finite clouds of particles (as the number of particles will keep growing with $t$ and as such will not be preserved) so we will need to look for fixed points among infinite point processes. The convergence of BBM viewed from its frontier $\calE_t \to \calE_\infty$ leads us to the correct notion: indeed starting at any large time $t$, each particle in $\calE_t$ will keep evolving at later times $s>t$ independently of the other particles in $\calE_t$ as a BBM minus the deterministic drift 
\[
\hat m(s):= m(s)-m(t) = \sqrt{2}(s-t) - \frac{3}{2\sqrt{2}} \left[\log(s) -\log(t)\right]\,.
\]
This shows, as observed in \cite{ABK13}, that the effect of the log correction in $m(t)$ asymptotically flattens. Letting $t\to \infty$ this motivates the definition of 
the following infinite branching particle system. 

\subsection{Infinite branching particle system and invariant measures.}
At time $0$, we start from a locally finite point process $\theta$ (viewed as a random point in the space  $\calN$ of integer-valued locally finite measures, see Section \ref{ss.state} for more details on the state space/topology). For convenience, we write 
\[\theta = \sum_{i\in I}\delta_{x_i}\]
with $I$ a finite or countable index set. Note that the atoms $x_i$ need not to be distinct. Given $\theta$, for any atom $x_i$, we run an independent Branching Brownian motion started from $x_i$ with critical drift $-\sqrt{2}$ which is denoted by $(\{\chi^{i}_k(t)-\sqrt{2}t; 1\leq k\leq n^i(t)\},{t\ge0})$. Then at time $t\ge0$, we get the following point process
\[
\theta_t=\sum_{i\in I} \sum_{k=1}^{n^i(t)}\delta_{x_i+\chi_k^i(t)-\sqrt{2}t}\,, 
\]
with $\theta_0=\theta$.
Note that new particles keep being created as time increases while the negative drift helps preventing the system from exploding. 
A natural question is to find the fixed points of this branching particle system, i.e. the point processes $\theta$ such that 
\begin{equation}\label{fixedpoint}
\forall t>0, \theta_t \overset{d}= \theta_0.
\end{equation}
We call such point processes the \textit{fixed points of BBM with critical drift}. Equivalently, fixed points are probability measures $\pi$ on the Polish space $\calN$ (Section \ref{ss.state}) which are invariant under BBM with critical drift.
In the case where the drift is stronger (or {\em super-critical}), i.e. when BBM is shifted by $-\lambda t$ with $\lambda>\sqrt{2}$, the fixed points of BBM with super-critical drift $\lambda$ and with an assumption of locally finite intensity measure have been characterized in the work \cite{kabluchko} (see Sections \ref{ss.lit} and \ref{non-critical-drift}). For the critical case $\lambda_c=\sqrt{2}$, it was pointed out in (3.9) of \cite{ABK13} that the above limiting extremal process of branching Brownian motion $\calE_\infty$ gives an example of such a fixed point with critical drift. In this work we give a characterization of all possible fixed points $\theta\sim \pi$ when  $\lambda_c=\sqrt{2}$ and under the unique assumption that $\theta$ has a top particle a.s. (see the subspace $\M$ of $\calN$ in Section \ref{ss.state}). 

Let us add some words of caution here. To define properly the concept of fixed point, we face the following issue of coming down from $-\infty$: the fact the initial point process $\theta_0=\theta\in \calN$ does not necessarily imply that $\theta_t$ will still be locally finite for all $t>0$. The same issue of coming down from $-\infty$ issue is already present in the work
 \cite{liggett78} on {\em independent particle systems} (same negative drift $-\lambda$ but no branching) and which will be a constant source of inspiration through this paper (see Sections \ref{ss.lit} and \ref{ss.attempt}). We will therefore need to be more careful than in the above paragraph when we will define what we mean by an invariant measure for such a process. See in particular Definition \ref{d.FP}.

\subsection{Main result.}\label{ss.main}
To describe all fixed points of BBM with critical drift, we may state our main theorem by relying on the above limiting process $\calE_\infty$, but this would make our main statement quite technical as one would first need to extract the derivative martingale $Z_\infty$ out of the point-process $\calE_\infty$ and then ``quotient it out''. To avoid this we will instead introduce the following slightly simpler process $\bar\calE_\infty$ which has the same law as $\calE_\infty$ except it does not have the additional random translation from the derivative martingale which is inherent to $\calE_\infty$. (See Section \ref{ss.extreme} for the relationship between $\calE_\infty$, $\bar \calE_\infty$ as well as with the more standard re-centred point process $\tilde \calE_\infty$). 
The law of this decorated Poisson Point Process $\bar\calE_\infty$ can be described as follows. 

Let $\mathcal{P}=\sum_{i\ge1}\delta_{p_i}$ be a Poisson point process (PPP) with  intensity $\sqrt{2} e^{-\sqrt{2}x}dx$. For each atom $p_i$ of $\mathcal{P}$, we attach a point process $\mathcal{D}^i=\sum_{j\ge1}\delta_{\mathcal{D}_j^i}$ where $\calD^i$, $i\ge1$ are i.i.d. copies of certain point process $\calD$ (independent of $\mathcal{P}$) called the {\em decoration process.} It is a point process supported on $(-\infty,0]$ with an atom at $0$ and its precise law is described in (6.8) of \cite{aidekon13}. See also Section \ref{ss.extreme}.  $\bar\calE_\infty$ may now be defined as the following point process 
\begin{equation}\label{e.barE}
 \bar\calE_{\infty} := \sum_{i, j} \delta_{p_i + \mathcal{D}^i_j}\,.
 \end{equation}


\ni
Furthermore, for any point process $\theta$ and for any real-valued random variable $S$, we define $\theta^S$ to be the point process $\theta$ shifted by $S$, i.e. if $\theta=\sum_{i\in I} \delta_{x_i}$,  
\begin{align*}\label{}
\theta^S:=\sum_{i \in I} \delta_{x_i+ S}\,.
\end{align*}

We are now ready to state our main result. 
\begin{theorem}\label{main-thm} 
Let $\theta$ be a point process so that $\theta\neq 0$ a.s. and $\theta(\R_+)<\infty$ a.s. (i.e. $\theta\in\M$ see~\eqref{e.M}) and 
let $\bar\calE_\infty$ be the point process defined in \eqref{e.barE}.
Then $\theta$  is a fixed point of BBM with critical drift (in the sense of Definition \ref{d.FP}) if and only if there exists a real-valued random variable $S$ independent of $\bar\calE_\infty$ such that
\begin{equation}\label{result} 
\theta \overset{d}= \bar\calE_\infty^S.
\end{equation}
\end{theorem}

\begin{remark}\label{}
Note that besides the assumption $\theta(\R_+)<\infty$ a.s, we do not make any hypothesis on the integrability properties of the possible fixed points (in particular we make no assumption on the growth of the number of particles in $\theta$ at $-\infty$).
\end{remark}

\begin{remark}
Our proof will show that the random shift $S$ comes from the following convergence in law which holds for any fixed point $\theta$: 
\[
\int_{-\infty}^0 \P(M_t \ge \sqrt{2}t-x)\theta(dx) \xrightarrow{d} e^{\sqrt{2}S}.
\]
\end{remark}
 
\begin{remark}\label{}
The if part in the theorem, i.e. the {\em existence} part has been proved in \cite{ABK13,aidekon13} at least in the case of the limiting process $\calE_\infty$. The proof in Section 3.2. of \cite{ABK13} is very short and combines the convergence in law $\calE_t \to \calE_\infty$ with the observation that the $\log$ correction term flattens as $t\to \infty$. Thanks to this flattening, the effective drift felt by the particles becomes asymptotically linear, i.e. $-\sqrt{2} t$. We claim that by a slight modification of the arguments in \cite{ABK13}, it is not hard to obtain the invariance of the process $\bar \calE_\infty$ as well (and thus of all our fixed points $\bar\calE_\infty^S$).   
\end{remark}

\begin{remark}\label{}
Note that the empty (or zero) point process is also a natural fixed point of the BBM particle system with critical drift. For example all finite point processes as well as all infinite point processes whose intensity does not blow up sufficiently fast on $\R_{-}$ are in the basin of attraction of $0$. It is an instructive exercice to check that the following examples of point processes asymptotically converge to 0. (In the space $\calN$ equipped with the vague topology, see Section \ref{ss.state}). 
\bnum
\item The deterministic point process $\theta_0=\sum_{i\in \N} \delta_{-i}$. 
\item The point process $\theta_0\sim \PPP(1_{(-\infty,0]} dx)$. 
\item The point process $\theta_0 \sim \PPP(e^{-\lambda x} dx)$ for any $0<\lambda \leq \lambda_c=\sqrt{2}$. The critical case is true but less easy.
\enum
\end{remark}
 
\begin{remark}
The question of the {\em basins of attractions} is interesting in its own. We shall only initiate the study of this question in Section \ref{s.basin} by designing a large space $\M_{3/2}\subset \M \subset \calN$ which is shown to contain the fixed points from Theorem \ref{main-thm} and which is built in such a way that it prevents any coming down from $-\infty$. See Section \ref{s.basin}. 
\end{remark}


\begin{remark}
If $\theta$ is a fixed point, it is immediate that a superposition of $n$ i.i.d. copies of $\theta$ is also a fixed point. This is linked to the fact that $\calE_\infty$ satisfies the invariance property under superpositions (see Corollary 3.3 of \cite{ABK13}). Moreover, Maillard \cite{maillard-stability} showed the equivalence between the invariance property under certain superpositions and the structure of decorated Poisson point process with exponential intensity.
\end{remark}

\subsection{Links to other works.}\label{ss.lit}
In this section we briefly make a few links with other related works in the literature.
\bnum
\item  In the work \cite{liggett78}, Liggett focuses on the case of {\em independent particle systems} where particles evolve independently of each other according to Brownian motions with negative drift $-\lambda t$. (His work applies to more general Markov processes than Brownian motion). We are thus considering the same particle system as Liggett except particles in our case are also subject to branching.  Liggett's work has been very influential in the recent years especially since the work of Biskup and Louidor \cite{BL16} where the fact that local extrema of a Discrete Gaussian Free field are asymptotically distributed as a shifted Poisson Point Process with intensity $e^{-\lambda x} dx$ is extracted from Liggett's theorem \cite{liggett78} thanks to a beautiful ``Dysonization'' procedure. See also \cite[Chapter 9]{biskupBook} for a very nice account on the characterization by Dysonization as well as \cite{zeitouni2017} where such a Dysonization procedure is also used.

\item The characterization theorem of Liggett has been further extended in the works by Ruzmaikina-Aizenman and Arguin-Aizenman \cite{AR05,AA09} where, motivated by links with spin glasses, they characterize the fixed-points of independent particle systems viewed modulo global translations. The analogous extension in our present setting, i.e. the fixed points for BBM with critical drift and modulo translations reveals some interesting re-shuffling properties of the fixed points in Theorem \ref{main-thm}. See our discussion in Section \ref{ss.aizenman}. 

\item Kabluchko analyzed in \cite{kabluchko} the fixed points of BBM under super-critical drift $\lambda>\sqrt{2}$ and under an assumption of locally finite intensity measure. See the discussion in Section \ref{non-critical-drift}. 

\item As mentioned above, Maillard characterized in \cite{maillard-stability} the point processes which are invariant under superposition (or more precisely point processes which are {\em exp-1-stable}). The difference with our present work is two-fold:  in our present setting we do not know a priori that all the fixed points in Theorem \ref{main-thm} have to  be {\em exp-stable}.
And the second difference is that the characterization in \cite{maillard-stability} does identify the Poisson Point Process with exponential intensity for the {\em leaders} but does not not characterize what is the decoration law as  {\em exp-1-stable} forms a large family of point processes.

\item 
We expect this analysis of fixed points to hold also for {\em branching random walks} (BRW) which are analogs of BBM in the discrete time (with a greater variety of displacement laws), see the book \cite{shi-book} for an introduction to Branching Random Walks and \cite{BK05,Aid13, Mad16} for relevant works. We discuss this question in Section \ref{ss.BRW}  where we highlight the fact that the question of fixed points also applies  to the so-called {\em lattice-BRWs} as opposed to the classical convergence results which may fail for lattice-BRWs (see \cite{Aid13,shi-book}).

\item The work \cite{BCM18} by Bertoin, Cortines and Mallein characterizes (under mild conditions) another natural family of branching processes called the  {\em branching stable} point processes. These are point processes $\mathbf{S}_1$ which satisfy the identity in law $\mathbf{S}_n \overset{d}= a(n) \mathbf{S}_1$ for some (deterministic) sequence $a(n)$ and where  $(\mathbf{S}_n)_{n\geq 1}$ denotes a branching random walk with reproduction law $\mathbf{S}_1$, starting from $\mathbf{S}_0:=\delta_0$.


\item Finally, such a characterization of fixed points may also be of interest for other natural point processes on $\R$. For example in the context of random matrices and determinantal processes, Najnudel and Virag introduced recently in \cite{najnudel} a family of infinite dimensional Markov chains (with an explicit transition mechanism) which are aimed at preserving the celebrated $\mathrm{Sine}_\beta$ point processes for any $\beta>0$. It is then a natural question to ask whether the $\mathrm{Sine}_\beta$ processes are the only invariant measures  for these Markov chains. 
\enum

\subsection{An Attempt of Proof of uniqueness.}\label{ss.attempt}

As we mentioned above, although our question is more difficult because of the branching setting, we have been inspired by the ideas of \cite{liggett78}. To make a comparison and to emphasize on the new ideas of this current paper, let us first recall a brief summary of \cite{liggett78}.

\subsubsection{Summary of Liggett's proof.}  (See also our companion paper \cite{liggett-new}). 
In the model of \cite{liggett78}, each points of a point process $\eta$ move independently of each other as a Markov chain on some state space $\mathcal{S}$. (Without great loss of generality, one may think of the case $\mathcal{S}=\R$ here).  Let $P(x,dy)$ denote the transition probabilities of this Markov chain. Let $\eta_n$ denote the point process at time $n$. To characterize fixed points $\eta$ such that $\eta_n \overset{d}{=}\eta$, it suffices to check that for all non-negative compactly supported functions $f\in C_c^+(\mathcal{S})$, 
\begin{equation}\label{eqn-for-eta}
\E[e^{-\langle f,\eta_n \rangle}] = \E[e^{-\langle f,\eta \rangle}],    
\end{equation}
where
\[\langle f,\eta \rangle := \int_{\mathcal{S}} f(x)\eta(dx).\]
Basic computations bring us to 
\begin{equation}\label{exact-comp}
\E[e^{-\langle f,\eta_n \rangle}] = \mathbb{E}\biggl[ \exp\biggl\{ \int_{\mathcal{S}} \log\bigl(\langle e^{-f(\cdot)},P^n(x,\cdot) \rangle\bigr) \eta(dx) \biggr\}\biggr].
\end{equation}
Under a uniform transience assumption that for all compact set $C\subset\mathcal{S}$,
\[\lim_{n\to \infty} \sup_{x} P^n(x,C) =0,\] 
one sees that as $n\to\infty$,
\begin{equation}\label{replace-log}
    \log\bigl(\langle e^{-f(\cdot)},P^n(x,\cdot) \rangle\bigr) = (1+ o_n(1))\langle e^{-f(\cdot)}-1,P^n(x,\cdot) \rangle.
\end{equation}
It then follows from Fubini theorem that
\begin{align}\label{factorisation}
    \E[e^{-\langle f,\eta \rangle}] = &\mathbb{E}\biggl[ \exp\biggl\{ (1+ o_n(1))\int_{\mathcal{S}}(e^{-f(y)} - 1)\mathbf{M}_n(dy) \biggr\}\biggr]\nonumber\\
    =&\E\left[\exp\left\{-(1+o_n(1)) \<{h, \mathbf{M}_n}\right\}\right]
\end{align}
where $\mathbf{M}_n(dy) := \langle P^n(\cdot,dy), \eta(\cdot) \rangle$ and $h(y)=1-e^{-f(y)}$. Let 
\[
C_c^b(\mathcal{S}):=\{h=1-e^{-f}| f\in C_c^+(\mathcal{S})\}.
\]
As a result, $\lim_{n\to \infty} \E[e^{-\<{h,\mathbf{M}_n}}]$ exists for any $h\in C_c^b(\mathcal{S})$. Note that $C_c^b(\mathcal{S})$ contains all compactly supported continuous function $h$ such that $0\leq h(y)<1$. Therefore this space is large enough to conclude that $\mathbf{M}_n(dy)$ converges in law to a random measure $\mathbf{M}_\infty(dy)$. So, we have
\begin{equation}
    \E[e^{-\langle f,\eta \rangle}] = \mathbb{E}\biggl[ \exp\biggl\{ \int_{\mathcal{S}}(e^{-f(y)} - 1)\mathbf{M}_\infty(dy) \biggr\}\biggr],
\end{equation}
which in turn means that $\eta$ must be a mixed Poisson point process with the random intensity measure $\mathbf{M}_\infty(dy)$. Using the constraint \eqref{eqn-for-eta} once more with $n=1$, it follows that $\mathbf{M}_\infty P\overset{d}{=}\mathbf{M}_\infty$. Under some additional assumptions on the underlying Markov chain, this implies that $\mathbf{M}_\infty P =\mathbf{M}_\infty$ a.s., which is the celebrated convolution equation of Choquet-Deny. It can be solved explicitly in many situations, see \cite{Deny} and \cite{ChoquetDeny}. As alluded to above, we remark that this approach has also been a key ingredient in the work of Louidor and Biskup \cite{BL16} on the convergence of extreme values of discrete Gaussian free field (DGFF). 

\subsubsection{Liggett's proof and Choquet-Deny equation applied to the BBM.}
To obtain the characterisation of BBM fixed points, a natural first attempt is to implement Liggett's strategy to the case of BBM.  Recall that $\theta_t=\sum_{i\in I}\sum_{1\leq k\leq n^i(t)}\delta_{x_i+\chi_k^i(t)-\sqrt{2}t}$. Observe that for any continuous function supported in a compact set $f\in C_c^+(\R)$, one has
\begin{equation}\label{Laplace-time-t}
\E[e^{-\langle f,\theta_t \rangle}] = \E\biggl[\exp\biggl(\int_{\R}\log\E[e^{-\sum_{k=1}^{ n(t)}f(x+ \chi_k(t)-\sqrt{2}t)}]\theta(dx)\biggr)\biggr].
\end{equation}
Next, by introducing $\mathcal{D}_t := \sum_{k\leq n(t)} \delta_{\chi_k(t)-M_t}$, one gets
\begin{align*}
\log\E[e^{-\sum_{k=1}^{ n(t)}f(x+ \chi_k(t)-\sqrt{2}t)}] = &\log\E[e^{-\<{f(x+ M_t-\sqrt{2}t+ \cdot), \mathcal{D}_t}}]\\
=&-(1+o_t(1))(1-\E[e^{-\<{f(x+ M_t-\sqrt{2}t+ \cdot), \mathcal{D}_t}}])
\end{align*}
where the second equality holds for fixed $x\in\R$ since $\theta\in\M$ (i.e. with a finite mass on $\R_+$ a.s.). Note that as opposed to Liggett's case, the term $o_t(1)$ is no longer uniform in $x$ here. Besides this first complication (which will not be a major one), one may expect that when one conditions $M_t-\sqrt{2}t$ to be large (so that $f(x+M_t-\sqrt{2} t)$ does not vanish), then $M_t-\sqrt{2}t$ and $\mathcal{D}_t$ will asymptotically be independent and that $\calD_t$ should converge in law as $t\to \infty$ to the limiting decoration process $\calD$ we have seen earlier in the definition of $\bar \calE_\infty$ in~\eqref{e.barE}. 
If so, this would lead us to 
\[
\E[e^{-\<{f(x+ M_t-\sqrt{2}t+ \cdot), \mathcal{D}_t}}] \approx \int_{\R} \E[e^{-\<{f(y+ \cdot), \mathcal{D}}}]P^t(x,dy),
\] 
where $P^t(x,dy)$ denotes the (non-Markovian) transition probabilities of $M_t-\sqrt{2}t$. Similarly as above, we set $H(y):= 1- \E[e^{-\<{f(y+ \cdot), \mathcal{D}}}]$ and obtain that
\[
\E[e^{-\langle f,\theta \rangle}] \approx\E\left[\exp\left(-(1+o_t(1))\int_{\R}H(y) \underbrace{\int_{\R} P^t(x,dy)\theta(dx)}_{\mathbf{M}_t(dy)}\right)\right]
\]
Nevertheless, because of the presence of the limiting decoration process $\mathcal{D}$, it seems far from obvious to check that the class of such functions $H$ is sufficiently large to characterize the convergence in law of $\mathbf{M}_t(dy)$ to a limiting random measure $\mathbf{M}_\infty(dy)$. Let us be more specific and call $\Psi$ this class of functions, i.e.
\begin{align}\label{e.classPsi}
\Psi:= \{ H_f \text{ s.t.} f\in C_c^+(\R) \text{ and }\forall y\in \R, H_f(y):= 1- \E[e^{-\<{f(y+ \cdot), \mathcal{D}}}] \}\,.
\end{align}
We may thus summarize the three difficulties we would need to face if one would want to follow Liggett's strategy as follows:
\bnum
\item First, the error term in the $\exp(\int_{\R} \ldots \theta(dx))$ is a $o_{t,x}(1)$ rather than $o_t(1)$ and this error term degenerates for some $x$ at sufficiently large distance depending on $t$. 
\item Second, we need an asymptotic factorization of $M_t -\sqrt{2t}$ and $\calD_t$ plus the convergence of the later towards $\calD$. This will indeed be correct but only for initial atoms $x\in \theta$ in a window $[-\sqrt{t}/\delta,-\delta \sqrt{t}]$ and is not expected to be correct elsewhere.
\item Finally, probably the main technical issue is as follows: assuming issues (1) and (2) have been successfully addressed, then Liggett's strategy would bring us to the following identity on (possible subsequential scaling limits) $\mathbf{M}_\infty(dy)$ of the measures $\mathbf{M}_t(dy)$:
for any $H=H_f\in \Psi$, 
\begin{align*}\label{}
\Eb{\exp(- \<{H_f, \mathbf{M}_\infty}} = \Eb{e^{-\<{f,\theta}}}\,.
\end{align*}
Yet, as mentioned above, it would remain to show that the class of functions $\Psi=\{H_f, f\in C_c^+(\R)\}$ is large enough to characterize the limiting measure(s) $\mathbf{M}_\infty(dy)$. 
Note for example that the class $\Psi$ does not satisfy usual Stone-Weierstrass' type of hypothesis. 
\enum
Let us write this main difficulty as an open question.
\begin{question}\label{q.Psi}
Is the class of functions $\Psi$ defined in \eqref{e.classPsi} sufficiently large to characterize the law of a random positive Radon measure on $\R$ ?
\end{question}
To bypass this difficulty, we designed a new strategy. In fact, this new strategy can also be used to give a different proof of Liggett's theorem \cite{liggett78}. We present it in our companion paper \cite{liggett-new}.

\begin{figure}[!htp]
\begin{center}
\includegraphics[height=9cm]{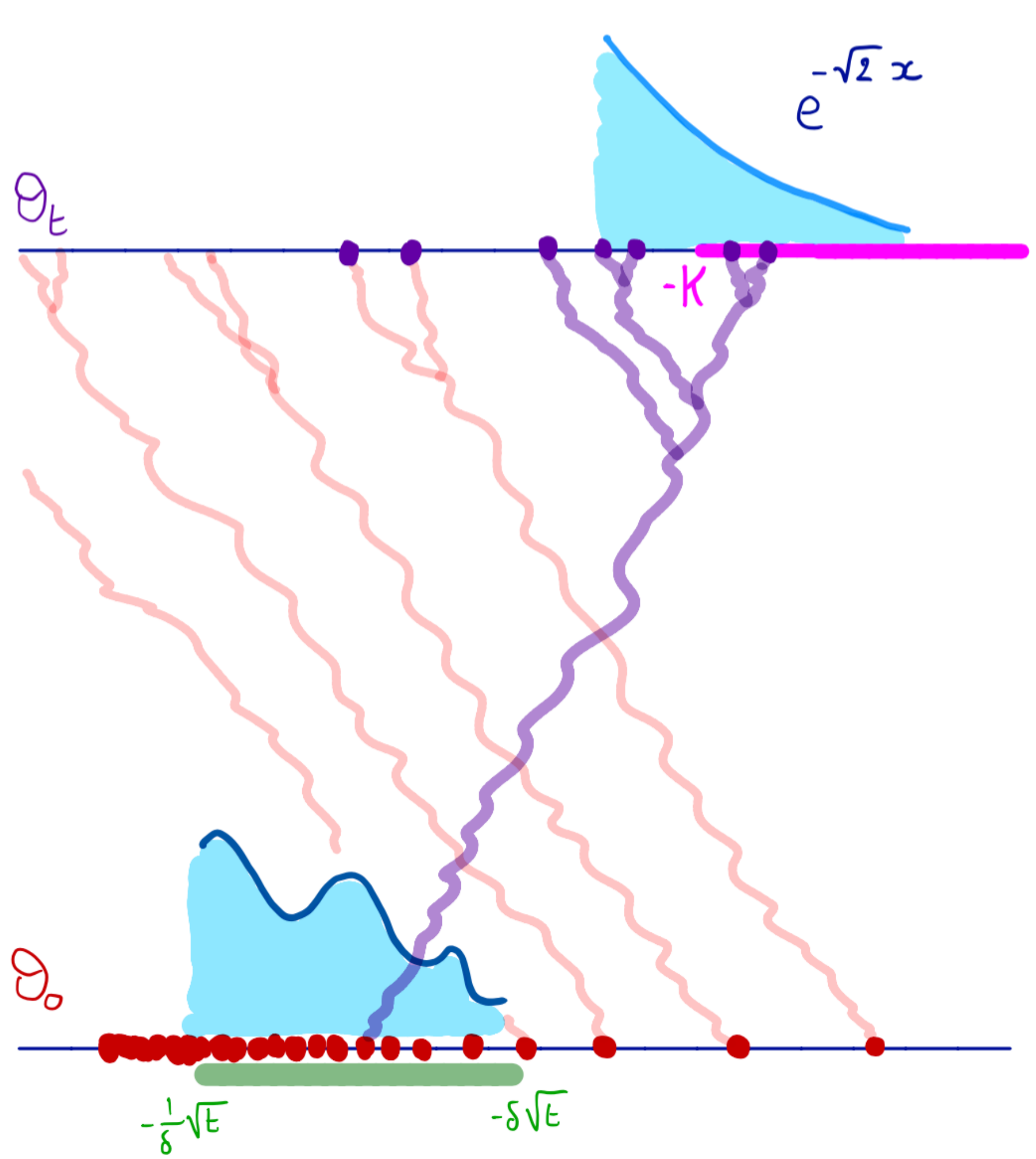}
\includegraphics[height=9cm]{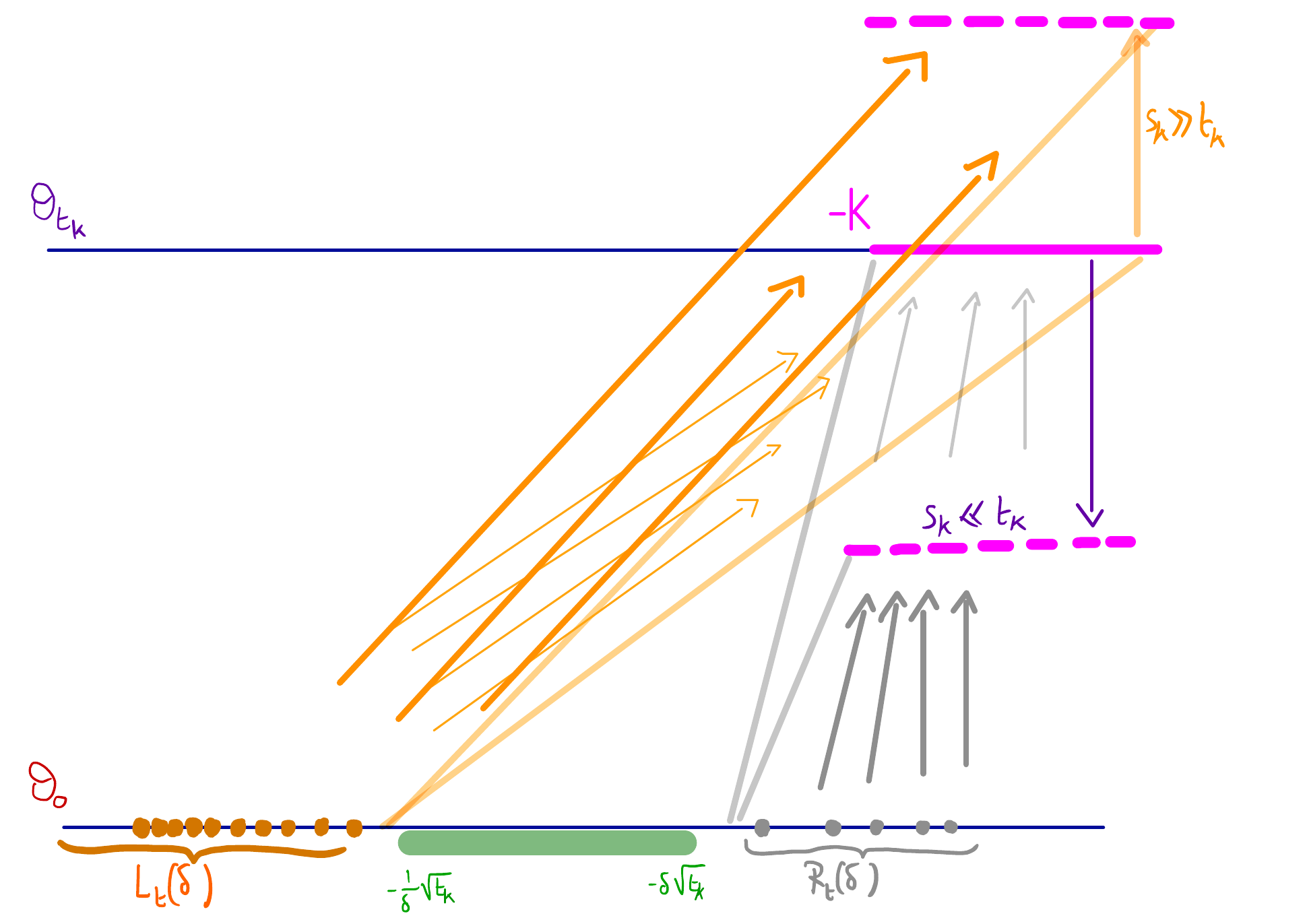}
\end{center}
\caption{An artistic view of the proof.}\label{f.sketch}
\end{figure}

\subsection{Heuristic ideas and sketch of proof.}\label{proof-sketch} 
We now outline the main ideas behind our proof. As in \cite{liggett78}, the equality in law $\theta_t\overset{d}= \theta_0$ for any $t>0$ will give us many equations which need to be satisfied by the fixed point $\theta$. We are going to focus in particular on the asymptotic behaviour of  the Laplace transforms \eqref{Laplace-time-t} along some well-chosen subsequence of times $\{t_k\}$ going to infinity. The limit of \eqref{Laplace-time-t} is obviously $\E[e^{-\<{f,\theta}}]$ as the point process $\theta_0=\theta$ is assumed to be a fixed point. Note that in \eqref{Laplace-time-t}, $f$ is supported in a compact set of $\R$. So, for some fixed $K\in \R_+$,
\begin{align*}
&1-\E[e^{-\sum_{k=1}^{ n(t)}f(x+ \chi_k(t)-\sqrt{2}t)}]\\
&=\E\left[(1- e^{-\sum_{k=1}^{ n(t)}f(x+ \chi_k(t)-\sqrt{2}t)})\ind{M_t-\sqrt{2}t+x\geq-K}\right]\\
&=\E\left[1- e^{-\<{f(x+M_t-\sqrt{2}t+\cdot), \calD_t}}\Big \vert M_t-\sqrt{2}t+x \geq -K\right]\P(M_t-\sqrt{2}t+x\ge -K).
\end{align*}
As $\theta(\R_+)<\infty$ a.s., for any atom $x$ of $\theta$ and any large $t>0$, we may approximate $\log \E[e^{-\sum_{k=1}^{ n(t)}f(x+ \chi_k(t)-\sqrt{2}t)}]$ by $-(1-\E[e^{-\sum_{k=1}^{ n(t)}f(x+ \chi_k(t)-\sqrt{2}t)}])$. (N.B. To control this approximation step uniformly over starting points $x$, we will need in the actual proof to introduce a large cut-off $A$ so that one focuses only on initial points $x\in(-\infty,A]$). Consequently, the convergence of \eqref{Laplace-time-t} along a subsequence will follow from the tightness of the following random variable as $t\to \infty$. 
\[
\int \E\left[(1- e^{-\<{f(x+M_t-\sqrt{2}t+\cdot), \calD_t}})\Big \vert M_t-\sqrt{2}t+x \geq -K\right]\P(M_t-\sqrt{2}t+x\ge -K)\theta(dx)
\]
Call the integrant function in the above integral $G_t(x)$.
At this point, in the above  integral $\int_{\R} G_t(x) \theta(dx)$, we do not know a priori which parts in space will contribute most to this integral (as we do not yet know what is the structure underlying the point process $\theta$). Now comes the main observation in the proof:  we make in some sense an {\em educated guess.} Namely, if all fixed points happen to have the expected structure given by the decorated point process $\bar\calE_\infty$ (plus drift), then the precise quantitative results from \cite{ABK13,aidekon13} tell us that the above integral should be very well approximated by points $x$ coming from the window $[-\tfrac 1 \delta \sqrt{t}, -\delta\sqrt{t}]$, where $\delta$ is chosen small enough (the smaller $\delta$ is, the better the approximation will be). In other words,  assuming the theorem indeed holds, we expect to have
\begin{align*}\label{}
\int  G_t(x) \theta(dx)
&\approx 
\int_{[-\tfrac 1 \delta \sqrt{t}, -\delta \sqrt{t}]} G_t(x) \theta(dx)
\end{align*}
The great news with this is that these initial points $x$ are precisely the points for which we have a convergence and a decoupling result under the appropriate conditioning as $t \to \infty$ 
of $(M_t-\sqrt{2} t, \calD_t)$ to an exponential variable times the limiting decoration process. This asymptotic decoupling from \cite{ABK13} will be stated in Lemma \ref{new-lemma}. Recall 
\[
G_t(x):=  \E\left[1- e^{-\<{f(x+M_t-\sqrt{2}t+\cdot), \calD_t}}\Big \vert M_t-\sqrt{2}t+x \geq -K\right]\P(M_t-\sqrt{2}t+x\ge -K)
\]
As such, Lemma \ref{new-lemma} will give us that for points $x$ in this window $[-\tfrac 1 \delta \sqrt{t}, -\delta\sqrt{t}]$, one has as $t\to \infty$
\begin{align*}\label{}
G_t(x)\approx  \P(M_t-\sqrt{2}t+x\ge -K)
\Big(\int_{-K}^\infty \sqrt{2}\E\left[(1-e^{-\int f(y+z)\mathcal{D}(dz)})\right] e^{-\sqrt{2}y-\sqrt{2}K}dy\Big)
\end{align*}
The advantage of this expression is two-fold: first we see the limiting expected  structure appearing, and second (up to a small error) the dependence in the point $x$ in $G_t(x)$ is now reduced to the probability $\P(M_t-\sqrt{2}t+x\ge -K)$.
So far notice that we have not yet used the fact that $\theta$ is a fixed point. Here comes its first key use: by using that $\theta_t\overset{d}=\theta$, one argues that the random variables 
\[
\int_{[-\tfrac 1 \delta \sqrt{t}, -\delta \sqrt{t}]} P(M_t-\sqrt{2}t+x\ge -K) \theta(dx) 
\]
need to be tight as $t\to \infty$ (otherwise the point process $\theta_t$ would need to blow up in, say the window $[-K,K]$).
This first use of our assumption leads us to the fact that (up to some work on the dependance on the width $K$ of the support of $f$) 
\begin{align*}\label{}
\int_{[-\tfrac 1 \delta \sqrt{t}, -\delta \sqrt{t}]} G_t(x) \theta(dx) 
\end{align*}
converges under subsequences $t_k \to \infty$ to the desired structure.

\medskip
At this point, we are still left with the main step of the proof which consists in showing that we were indeed allowed to make the above {\em educated guess}. Namely, it remains to prove that when $\delta$ is small, points $x$ outside of the window $[-\tfrac 1 \delta \sqrt{t}, -\delta \sqrt{t}]$ cannot contribute significantly to $\int G_t(x) \theta(dx)$. We will for this analyze the following Left and Right terms:
\begin{align*}\label{}
\begin{cases}
L_t(\delta):= \int_{(-\infty,-\tfrac 1 \delta \sqrt{t}]} G_t(x) \theta(dx) \\
R_t(\delta):= \int_{[\delta \sqrt{t},A]} G_t(x) \theta(dx)\,,
\end{cases}
\end{align*}
where $A$ is the cut-off mentioned above. (Also our definitions of Left and Right terms will slightly differ in the actual proof). We will show that these two random variables converge in probability to zero as $t\to \infty$ and then $\delta\to 0$ and we will proceed in both cases by contradiction. In a few words, we will argue as follows
\bi
\item[i)] For the right term $R_t(\delta)$, assume by contradiction that one can find a sequence $t_k\to \infty$ such that $R_{t_k}(\delta)$ is bounded away from zero with positive probability. Then we will show that this leads to a contradiction by  looking at earlier times $s_k:= \delta^2 \, t_k \ll t_k$ for which we will detect an explosion for the number of points in $[-K,K]$ for the point processes $\theta_{s_k}$ (at least as $\delta \to 0$). The key quantitative estimate will be the uniform control over $x\in [-\delta\sqrt{t}, A]$ provided by Lemma \ref{ratio-R}. 
\item[ii)] For the left term $L_t(\delta)$, we also proceed by contradiction. For a similarly defined sequence $\{t_k\}$, we now obtain the contradiction by looking at the later times $s_k:= 2 t_k$. We will show also that the point processes $\theta_{s_k}$ will accumulate two many points in $[-K,K]$ if $L_t(\delta)$ does not converge to 0 in probability. The proof here will be much more delicate as we do not have such a uniform control as in Lemma \ref{ratio-R} for the right term. Instead we rely on precise estimates given by Bramson's $\psi$-function (which will be recalled in Section \ref{ss.bramson2}) This will be the purpose of Lemma \ref{bd-L}. 
\ei 
This ends our sketch of proof. See also Figure \ref{f.sketch} which serves as an illustration of the strategy implemented. 

\begin{remark}
As remarked earlier, the idea explained above can be adapted to give a new proof of the above mentioned Liggett's result. We implement this idea in a companion paper \cite{liggett-new}. The novelty in our approach is that it avoids using Choquet-Deny convolution equation (\cite{ChoquetDeny}, \cite{Deny}) and this seems to be necessary if one wants to avoid answering the seemingly tedious Question \ref{q.Psi}. 
\end{remark}

\subsection*{Organization of the paper.} The rest of this paper is organised as follows. In Section \ref{facts-BBM}, we recall some facts and results on BBM.
Uniqueness is proved in Section \ref{Uniqueness}. Section \ref{s.basin} introduces a space, $\M_{3/2}$, which is left invariant by BBM and should as such be relevant for the study of basins of attractions. In the final Section \ref{s.conclusion}, we discuss Kabluchko's results, the question of fixed points for BRWs (including the lattice-case) and the fixed points modulo translations in the spirit of \cite{AR05,AA09}.

\subsection*{Acknowledgements.}
We wish to thank Elie A\"id\'ekon for pointing to us the reference \cite{kabluchko} for the non-critical case. 
 The research of X. C. is supported by ANR/FNS MALIN.
The research of C.G. and A.S.  is supported by the ERC grant LiKo 676999.

\section{Preliminaries}\label{facts-BBM}
We first define what is the setup/state space and then we recall some facts on BBM as well as the related FKPP equation, mainly extracted from \cite{bramson78,bramson83,ABK13,aidekon13,bovier-book} and which will be used to prove the main theorem.

\subsection{State space.}\label{ss.state}
Let $\calN$ be the space of integer valued measures on $\R$ which are locally finite. 
As in the introduction, we represent any (deterministic) point $\eta\in \calN$ as follows  
\[
\eta = \sum_{i\in I}\delta_{x_i}
\]
with $I$ a finite or countable index set and where the atoms $x_i$ need not to be distinct. In the rest of this text, we will use the following notation convention: $\eta\in \calN$ will denote a (deterministic) point in $\calN$ while $\theta$ will in general denote a {\em point process}, i.e. a random variable in $\calN$. 
This space $\calN$ is naturally equipped with the vague topology, see \cite{kallenberg-book}.
\begin{remark}\label{}
Note that the weak topology is not appropriate for the processes we consider. This is due to the following reason: recall $\theta_n \overset{w}\longrightarrow \theta$ if and only if for any continuous bounded $f\in C_b(\R)$,
$\theta_n(f) \to \theta(f)$. But the processes we are interested in have a diverging mass near $-\infty$, as such they will not integrate, say the continuous function $f\equiv 1$. The vague topology is more indulgent and corresponds instead to $\theta_n\overset{v}\longrightarrow \theta$ if and only if for any $f\in C_c(\R)$, $\theta_n(f) \to \theta(f)$.  
\end{remark}
The vague topology on $\calN$ is metrizable and one can define a metric $d=d_{\calN}$ on $\calN$ such that the space $(\calN,d)$ is Polish (see Theorem A2.3 in \cite{kallenberg2006}). As such one may now consider probability measures on $\calN$ in the usual way. 

We also introduce the following key subspace:
\begin{align}\label{e.M}
\M:= \{ \eta\in \calN, \eta([0,\infty)) <\infty \} \subset \calN\,.
\end{align}
This means that point processes which are a.s. in $\M$ have a.s. a top particle and may then also be viewed as  non-increasing sequences $x_1\geq \ldots \geq x_n \geq \ldots$. Note though that the space $\M$ is not closed in $(\calN,d)$. Due to this,  our state space will still be the Polish space $\calN$ but the probability measures $\pi$ we consider in Theorem \ref{main-thm} are probability measures on $\calN$ such that if $\theta\sim \pi$ then $\theta\in \M$ a.s.

\subsection{Laplace functional and notion of invariant measure.}

It is well known that if $\theta$ is point process in $\calN$, then its law is characterised by the Laplace functional $\Psi_\theta$ defined on the set of non-negative Borel measurable functions as follows:
\[
\Psi_\theta(f):=\E[e^{-\<{f,\theta}}], \ \forall f : \R\to\R_+ \textrm{ Borel measurable}.
\]
Usually, we take $C_c^+(\R)$, the class of non-negative, continuous and compactly supported functions, which are sufficient to determine the law of point process. See \cite{kallenberg-book}. However, in this paper we are interested in point processes which live a.s. in the space $\M$ defined in~\eqref{e.M}, i.e.  which are such that $\theta(\R_+)<\infty$ a.s. Therefore, as explained for example in \cite{bovier-book} Chapter 7.2, we may also consider the class of functions $f$ of the form 
\begin{equation}\label{f-type} 
f(x)= \sum_{k=1}^n c_k 1_{x >b_k}, \textrm{ with } c_k>0, b_k\in\R.
\end{equation}
It is not hard to check  that the class of functions of this form is sufficiently rich to characterise the law of a point process in $\M$ 
 (see e.g. \cite{kallenberg-book} for more details). This class of functions will be important for us in order to be able to rely on Bramson's $\psi$-function introduced  in Section \ref{ss.bramson2}.
\begin{remark}
Note that this class of functions is only suitable to determine processes in $\M$. Otherwise, if we take two Poisson point processes with different intensities $e^xdx$ and $e^{2x}dx$, we have $\Phi_1(f)=\Phi_2(f)=0$ for any function $f$ in this class.
\end{remark}

Let us now state the following Lemma which will allow us to give a rigorous meaning to the notion of fixed point/invariant measure. 
\begin{lemma}\label{l.Lthetat}
For any probability measure $\pi$ on $\calN$ and for any function $f$ either in $C_c^+(\R)$ or of the form~\eqref{f-type}, the Laplace transform $\Eb{e^{-\<{f,\theta_t}}}$ is well defined even if a coming down from $-\infty$ happened from the point process  $\theta_0=\theta$ to $\theta_t$. 
\end{lemma}
\ni
{\em Proof.}
Let $f$ be a function as in the Lemma and fix some $t>0$. Let us assume without loss of generality that there are countably many points in the point process $\theta\sim \pi$, and let us write $\theta$ as  
\[
\theta = \sum_{i=1}^\infty \delta_{x_i}\,.
\]
By definition of the BBM process with critical drift $t\mapsto \theta_t$, we have 
\begin{align*}
\Eb{e^{-\<{f,\theta}}}:= \int_{\calN} d\pi(\theta) e^{-\<{f,\theta}}& = \int_{\calN} d\pi(\theta) \EFK{}{\theta}{e^{-\<{f,\theta_t}}}  \\
& =  \int_{\calN} d\pi(\theta) \liminf_{L\to \infty} \Eb{ \prod_{i=1}^L \Eb{ \prod_{k=1}^{n^i(t)} e^{-\<{f, \delta_{x_i+\chi_k^i(t)-\sqrt{2}t}}} } } \\
& =: \Eb{e^{-\<{f,\theta_t}}}\,. 
\end{align*}
This shows that the Laplace transforms $\Eb{e^{-\<{f,\theta_t}}}$ are well defined even if a coming down from $-\infty$ occurs (in which case these Laplace transforms will be identically zero as soon as $f\not\equiv 0$).
\qed 

We may now give a precise definition of the invariant measures for the BBM process with critical drift.  
\begin{definition}[Fixed points / invariant measure]\label{d.FP}
A probability measure $\pi$ on $\calN$ is said to be an invariant measure for the BBM with critical drift if when starting from $\theta_0=\theta \sim \pi$, one has for any $t>0$, $\theta_t \sim \pi$.  This is equivalent to asking that for any $f\in C_c^+(\R)$,
\[
\Eb{e^{-\<{f,\theta}}} = \Eb{e^{-\<{f,\theta_t}}}\,,
\]
which is well defined thanks to Lemma \ref{l.Lthetat}. 

If furthermore the probability measure $\pi$ is supported on $\M$ (since $\M$ is not closed, we mean here that if $\theta\sim \pi$, then $\theta\in \M$ a.s.), then we will say that $\pi$ is an invariant measure for the BBM with critical drift if and only if for any function $f$ of the form~\eqref{f-type}, one has 
\[
\Eb{e^{-\<{f,\theta}}} = \Eb{e^{-\<{f,\theta_t}}}\,.
\]
We call these measures $\pi$ the  fixed-points of BBM with critical drift. 
\end{definition}

\begin{remark}\label{}
Note that our main Theorem \ref{main-thm} characterizes all fixed points $\pi$ supported in $\M$ (in the above sense) but does not exclude a priori the existence of exotic fixed points on $\calN$. See Question \ref{q.calN}.
\end{remark}

\subsection{Extremal process of BBM.}\label{ss.extreme}

For a binary  BBM $(\{\chi_k(t); 1\leq k\leq n(t)\}, t\ge0)$ on the real line, recall we denote the maximal position at time $t$ by
\[
M_t:=\max_{1\leq k\leq n(t)}\chi_k(t). 
\]
Bramsom \cite{bramson78}, \cite{bramson83} proved that $M_t-m(t)$ converges in law to some non-degenerate random variable, where 
\begin{align}\label{e.m}
m(t):= \sqrt{2}t - \frac{3}{2\sqrt{2}}\log_{+}(t).
\end{align}
Then, Lalley and Sellke showed in \cite{lalley-sellke} that the limiting distribution function is $\E[e^{-\Cb_MZ_\infty e^{-\sqrt{2}x}}]$ where $\Cb_M$ is some positive constant which will appear again in \eqref{omega} and $Z_\infty\in(0,\infty)$ is the a.s. limit of the so-called \textit{derivative martingale}
\[Z_t := \sum_{k=1}^{n(t)}(\sqrt{2}t- \chi_k(t))\exp\bigl\{-\sqrt{2}(\sqrt{2}t- \chi_k(t))\bigr\}.\]
Later, it was proven in \cite{ABK13} and \cite{aidekon13} that the point process defined by
\begin{equation}\label{abk-conv}
    \mathcal{E}_t := \sum_{k=1}^{n(t)}\delta_{\chi_k(t)- m(t)} 
\end{equation}
converges to a non-trivial point process $\mathcal{E}_{\infty}$ as $t\to \infty$, in the sense of vague convergence of distributions on the space $(\calN, d)$. The point process $\mathcal{E}_\infty$ is called the \textit{(limiting) extremal point process of BBM}. The law of $\mathcal{E}_\infty$ can be described as follows. 

Let $\mathcal{P}=\sum_{i\ge1}\delta_{p_i}$ be a Poisson point process (PPP) independent of $Z_\infty$ and with intensity $\sqrt{2}\Cb_Me^{-\sqrt{2}x}dx$. For each atom $p_i$ of $\mathcal{P}$, we attach a point process $\mathcal{D}^i=\sum_{j\ge1}\delta_{\mathcal{D}_j^i}$ where $\calD^i$, $i\ge1$ are i.i.d. copies of certain point process $\calD$ and independent of $(\mathcal{P},Z_\infty)$. In this way, we get
\begin{equation}\label{def-Et}
 \mathcal{E}_{\infty} = \sum_{i, j} \delta_{p_i + \mathcal{D}^i_j + \frac{1}{\sqrt{2}}\log(Z_\infty)}.
 \end{equation}
The point process $\mathcal{E}_\infty$ is thus called a decorated Poisson point process with decoration process $\calD$. Moreover, the decoration process $\calD$ is a point process supported on $(-\infty,0]$ with an atom at $0$ and its precise law is described in (6.8) of \cite{aidekon13}. In \cite{ABK13}, it is shown that $\mathcal{E}_\infty$ is a fixed point of BBM with critical drift.

As we have seen in Section \ref{ss.main}, it is often convenient to remove the randomness coming from the derivative martingale $Z_\infty$. The standard option is to consider the following point process:
\begin{equation}\label{tilde-def}
\widetilde{\mathcal{E}}_\infty := T_{- \frac{1}{\sqrt{2}}\log(Z_\infty)}\mathcal{E}_\infty=\sum_{i, j} \delta_{p_i + \mathcal{D}^i_j},
\end{equation} 
where the operator $T_u$ acts on point processes $\theta$ by shifting each atom by $u$. 
This process  was proved in \cite{aidekon13} (also in \cite{ABK13} but it is not explicitly stated there) to be limit in law of 
\[\widetilde{\mathcal{E}}_{t} := \sum_{k=1}^{n(t)}\delta_{\chi_k(t)- m(t)- \frac{1}{\sqrt{2}}\log(Z_t)}.\]

\begin{remark}\label{}
Recall that we stated Theorem \ref{main-thm} with the help of the point process $\bar\calE_\infty$. It is easy to check that $\tilde \calE_\infty$ 
is nothing but a deterministic shift of $\bar \calE_\infty$.  As such Theorem \ref{main-thm} also holds if one replaces $\bar\calE_\infty$ by $\tilde \calE_\infty$ (we have chosen the less standard  $\bar \calE_\infty$ in Section \ref{ss.main} as we did not need to introduce the constant $\Cb_M$). 
\end{remark}

\subsection{BBM and FKPP equation.} \label{prem-FKPP}

For the binary BBM, let us write $\calB_t:=\sum_{k=1}^{n(t)}\delta_{\chi_k(t)}$ its associated point process at time $t$. As $n(t)\in \N^*$ a.s. for any $t>0$, $\calB_t\in \M$ a.s. and for any function $f$ of the form \eqref{f-type}, we have
\[
\Phi_{\calB_t}(f)=\E[e^{-\<{f,\calB_t}}]=\E\left[\prod_{k=1}^{n(t)}e^{-f(\chi_k(t))}\right].
\]
In particular, the distribution function of the maximal position $M_t$ is $\P(M_t\le x)=\E[\prod_{k=1}^{n(t)}\ind{\chi_k(t)\le x}]$ which formally corresponds to $f(\cdot):=\infty 1_{\cdot \geq x}$.

Now let us state the following Lemma which highlights the connection between BBM and the Fisher-Kolmogorov-Petrovsky-Piscounov (FKPP) equation. This  was observed by McKean \cite{McKean75} and appeared also in \cite{Skorohod64} and Ikeda, Nagasawa, and Watanabe \cite{Ikeda-Nagasawa-Watanabe1}, \cite{I-N-W2}, \cite{I-N-W3}.

\begin{lemma}[See Lemma $5.5$ in \cite{bovier-book}]\label{BBM=FKPP}
For any measurable function $\varphi:\mathbb{R} \to [0,1]$, let
\begin{equation}\label{u-def}
u(t,x) = 1 - \mathbb{E}[\Pi_{k=1}^{n(t)}\{1-\varphi(x- \chi_k(t))\}].
\end{equation}
Then $u$ solves the following FKPP equation
\begin{equation}\label{FKPP}
\partial_t u = \frac{1}{2}\partial_x^2 u + u- u^2.
\end{equation}
 with the initial condition $u(0,x) = \varphi(x)$. 
\end{lemma}

In the following, for convenience, we usually write $u_\varphi$ for the solution of FKPP equation with initial function $\varphi$. In particular, when $\varphi(x)=\ind{x\le0}$, we write $u_M(t,x)$ for the associated solution as $u_M(t,x)=\P(M_t\ge x)$.

Immediately, one sees that if $\varphi(x)=1-e^{-f(-x)}$, we get that $\Phi_{\calB_t}(f)=1-u_\varphi(t,0)$. Moreover, for the point process $\calE_t$ in \eqref{abk-conv} which is $\calB_t$ shifted by $-m(t)$, one has $\Phi_{\calE_t}(f)=1-u_\varphi(t,m(t))$. So, for any $x\in\R$, 
\[
u_\varphi(t, m(t)+x)=1-\Phi_{\calE_t}(f(\cdot-x))\,.
\]

\subsection{Bramson's convergence results on the solutions of FKPP.}\label{ss.bramson1}
Next, let us state a convergence result on the solutions of FKPP equation, due to Bramson \cite{bramson83}. (See also Theorem 4.2 of \cite{ABK13} for a more complete presentation). We only state here a partial result which will be sufficient for our proof. Recall also the definition of the function $t\mapsto m(t)$ in~\eqref{e.m}.

\begin{theorem}[\cite{bramson83}]\label{cvgFKPP}
Let $u_\varphi$ be a solution of the FKPP equation \eqref{FKPP} with initial condition $u(0,x)=\varphi(x)$ 
where $\varphi$ is measurable function satisfying the following conditions:
\begin{align}
(i)\hspace{1cm}& 0\leq \varphi(x) \leq 1; \label{cond1}\\
(ii)\hspace{1cm}&\mbox{For some} \hspace{1mm} v>0, L>0, N>0, \int_x^{x+N} \varphi(y)dy > v \hspace{2mm}\forall x \leq -L; \label{cond2}\\
(iii)\hspace{1cm}&\sup\{x\in\R | \varphi(x) > 0\} < \infty. \label{cond3}
\end{align}
Then as $t\to\infty$, uniformly in $x$, $u_\varphi(t,m(t)+x)$ converges to a positive function $\omega(x)$ where $\omega$ is the unique solution (up to a $\varphi$-dependent translation) of the equation $\frac{1}{2}\omega''+\sqrt{2}\omega'+\omega-\omega^2=0$. This limiting function $\omega$ is called the traveling wave.
\end{theorem}

Note that for any function $f$ of the form \eqref{f-type}, $\varphi(x)=1-e^{-f(-x)}$ satisfies the conditions \eqref{cond1}, \eqref{cond2}, \eqref{cond3}. We hence get the convergence of $\Phi_{\calE_t}(f)$. Another example is when $\varphi(x)=\ind{x\le0}$, this theorem shows that 
\[
u_M(t,x+m(t))=\P(M_t\ge m(t)+x)
\]
converges to some limit $\omega_M(x)$ which, according to \cite{lalley-sellke}, turns out to be $1-\E[e^{-\Cb_M Z_\infty e^{-\sqrt{2}x}}]$. Moreover, it is also known (see e.g. \cite{bramson83}) that there exists a constant $\Cb_M>0$ such that
\begin{equation}\label{omega}
\omega_M(x) \sim \Cb_Mxe^{-\sqrt{2}x} \hspace{2mm}\mbox{as}\hspace{2mm} x\to +\infty,
\end{equation}
Then with a little more effort, one sees from Theorem \ref{cvgFKPP} that for any $\varphi$ satisfying \eqref{cond1}, \eqref{cond2}, \eqref{cond3}, there exists some constant $C_\varphi>0$ such that
\[
\lim_{t\to\infty}u_\varphi(t, x+m(t))=1-\E[e^{-C_\varphi Z_\infty e^{-\sqrt{2}x}}].
\]
In fact, in view of Corollary 6.51 of \cite{bovier-book}, $C_\varphi=\lim_{x\uparrow\infty}\lim_{t\to\infty}\frac{e^{\sqrt{2}x}}{x}u_\varphi(t,x+m(t))$.

As mentioned above, we can take $\varphi(x)=1-e^{-f(-x)}$ for any $f$ of the form \eqref{f-type} and obtain that
\begin{equation}\label{cvgEt}
\lim_{t\to\infty}\E[e^{-\<{f(\cdot-x),\calE_t}}]=1-\lim_{t\to\infty}u_\varphi(t, x+m(t))=\E[e^{-\Cb(f) Z_\infty e^{-\sqrt{2}x}}]
\end{equation}
where we write $\Cb(f)$ for $C_\varphi$. {\em (N.B we use two different notations here for the same constant in order to avoid confusions. Indeed in \cite{ABK13}, $\Cb(f)$ is defined for $f$ not for the initial function $\varphi$).} We will need the following explicit expression for the constant $\Cb(f)$.
\begin{proposition}[Proposition 7.9 of \cite{bovier-book}]
For any $f$ of the form \eqref{f-type}, the positive  constant $\Cb(f)$ can be written as the following limit which exists
\begin{equation}\label{constant-f}
\Cb(f)=C_\varphi=\lim_{r\to\infty}\sqrt{\frac{2}{\pi}}\int_0^\infty u_\varphi(r, y+\sqrt{2}r)ye^{\sqrt{2} y} dy.
\end{equation}
(N.B. It is known that \eqref{cvgEt} and \eqref{constant-f} remain valid if we take $\varphi(x)=1-e^{-f(-x)}$ with  $f\in C_c^+(\R)$, see \cite{ABK13}). 
\end{proposition}

Recall that it was proved in \cite{ABK13} and \cite{aidekon13} that the point process $\calE_t$ converges in law to the point process $\calE_\infty$. In view of \eqref{cvgEt}, we then have the following Laplace functional for the limiting extremal process $\calE_\infty$.

\begin{proposition}[Proposition $3.2$ of \cite{ABK13}] \label{Lap-Et}
For any $f$ of form \eqref{f-type}, 
\[\E[e^{-\langle f,\mathcal{E}_\infty \rangle}] = \E[e^{-Z_\infty\Cb(f)}].\]
\end{proposition}
\ni
It turns out that the key constant $\Cb(f)$ (which describes the Laplace transform of $\calE_\infty$) can be expressed out of the constant $\Cb_M$ from~\eqref{omega} through the decoration process $\calD$ (defined in \eqref{def-Et}) as follows:
\begin{equation}\label{e.CfCm}
\Cb(f) = \sqrt{2}\Cb_M \int_{\R}\E[1-e^{-\<{f(\cdot+y), \calD}}]e^{-\sqrt{2}y}dy\,.
\end{equation}
We will explain how to derive this useful identity in the next subsection. (See also \cite{ABK13} where this explicit identity for $\Cb(f)$ is implicit). 

\subsection{Limiting decoration process $\calD$ and uniform control in ancestors.}
In this Subsection, we define what is the limiting decoration process $\calD$. A special attention will be given to uniform results obtained in \cite{ABK13} on starting (conditioned) BBM uniformly from initial points $x\in[-\tfrac 1 \delta \sqrt{t}, -\delta \sqrt{t}]$. These uniform results will be key to our approach.

Recall that $M_t=\max_{1\leq k\leq n(t)}\chi_k(t)$ is the maximum of a BBM starting at the origin at time 0. Let also  $\overline{M}_t:=M_t-\sqrt{2}t$ and $\calD_t=\sum_{k=1}^{n(t)}\delta_{\chi_k(t)-M_t}$ the decoration process at time $t$.

The next Lemma is a direct consequence of Theorem 3.4 and Corollary 4.12 of \cite{ABK13} (see also (4.109) of \cite{ABK13}). It serves both as a definition of the limiting decoration process $\calD$ which is the limit in law of $\calD_t$ and it also quantifies the fact there is a certain freedom in the choice of the initial position. 

\begin{lemma}[4.109 in \cite{ABK13}]\label{new-lemma} Let $f$ be of form \eqref{f-type}  such that its support is contained in $[-K_f, \infty)$ with some $K_f\in\R$. Then, for any fixed $\delta \in (0,1)$, uniformly for $x\in[-\frac{1}{\delta}\sqrt{t},-\delta\sqrt{t}]$, the following convergence holds
\begin{align}\label{dec-conv}
\nonumber &\lim_{t\to\infty}\E\biggl[(1-e^{-\int f(x+\overline{M}_t+z)\mathcal{D}_t(dz)})\Big\vert x+\overline{M}_t\geq -K_f \biggr] \\ 
&=\int_{-K_f}^\infty \sqrt{2}\E\left[(1-e^{-\int f(y+z)\mathcal{D}(dz)})\right] e^{-\sqrt{2}y-\sqrt{2}K_f}dy.
\end{align}
\end{lemma}

The next Lemma provides related uniform control on the solutions of the FKPP equations for initial points in the same window $[-\sqrt{t}/\delta, -\delta \sqrt{t}]$. 
Recall that $u_M(t,x)=\P(M_t\ge x)$ and $u_\varphi(t,x)= 1 - \mathbb{E}\left[\prod_{k=1}^{n(t)}\{1-\varphi(x- \chi_k(t))\}\right]$.
This Lemma follows from Proposition 4.3 and Lemma 4.5 of \cite{ABK13}. 
\begin{lemma}[\cite{ABK13}]\label{decoration-lemma}
Let $\varphi(x)=1-e^{-f(x)}$ with $f$ of form \eqref{f-type}. Let $u_\varphi$ be a solution of the FKPP equation \eqref{FKPP} with initial condition $u(0,x)=\varphi(x)$. Then, for any fixed $\delta\in(0,1)$, uniformly over $x\in [-\sqrt{t}/\delta, -\delta\sqrt{t}]$, we have the following convergence as $t\to \infty$,
\begin{equation*}
\frac{t^{\frac{3}{2}}e^{\frac{x^2}{2t}}}{(-x)e^{\sqrt{2}x}}u_\varphi(t, \sqrt{2}t -x) \to \Cb(f).
\end{equation*}
Moreover, 
\begin{equation} \label{phi/M}
\frac{u_\varphi(t, \sqrt{2}t -x)}{u_M(t, \sqrt{2}t -x)} \to \frac{\Cb(f)}{\Cb_M},
\end{equation}
uniformly over $x\in [-\sqrt{t}/\delta, -\delta\sqrt{t}]$.
\end{lemma}
We now explain how the identity~\eqref{e.CfCm} follows from the combination of the above two lemmas.
We start by rewriting the expectation on the left side of \eqref{dec-conv} using the FKPP solution $u_\varphi$. Observe for this that
\begin{align*}
&\E\left[(1-e^{-\int f(x+\overline{M}_t+z)\mathcal{D}_t(dz)})\Big\vert x+\overline{M}_t\geq -K_f \right]\\
&\hspace{20mm}=\frac{\E\left[(1-\prod_{k=1}^{n(t)}e^{-f(x+\chi_k(t)-\sqrt{2}t)})\ind{x+M_t-\sqrt{2} t\ge -K_f}\right]}{\P(M_t \ge \sqrt{2}t-x-K_f)}\\
&\hspace{20mm}=\frac{\E\left[(1-\prod_{k=1}^{n(t)}e^{-f(x+\chi_k(t)-\sqrt{2}t)})\right]}{\P(M_t \ge \sqrt{2}t-x-K_f)}
\end{align*}
as the support of $f$ is contained in $[-K_f,\infty)$. Then note that
\begin{align*}
\E\left[(1-\prod_{k=1}^{n(t)}e^{-f(x+\chi_k(t)-\sqrt{2}t)})\right]=&u_\varphi(t,\sqrt{2}t-x),
\end{align*}
with $\varphi(x)=1-e^{-f(-x)}$. Since furthermore $\P(M_t \ge \sqrt{2}t-x-K_f)= u_M(t,\sqrt{2}t-x-K_f)$, it follows that uniformly over $x\in[-\sqrt{t}/\delta, -\delta\sqrt{t}]$
\begin{multline}\label{conditionedcvg}
\E\left[(1-e^{-\int f(x+\overline{M}_t+z)\mathcal{D}_t(dz)})\Big\vert x+\overline{M}_t\geq -K_f \right]\\
=\frac{u_\varphi(t,\sqrt{2}t-x)}{u_M(t,\sqrt{2}t-x-K_f)}
\xrightarrow{t\to\infty}e^{-\sqrt{2}K_f}\frac{\Cb(f)}{\Cb_M }
\end{multline}
by use of Lemma \ref{decoration-lemma}. Comparing it with \eqref{dec-conv}, one sees that 
\[
\Cb(f)=\Cb_M\int_{-K_f}^\infty \sqrt{2}\E\left[(1-e^{-\int f(y+z)\mathcal{D}(dz)})\right] e^{-\sqrt{2}y}dy.
\]
It remains to replace $-K_f$ by $-\infty$ in the integral. This can be easily seen as follows. Note that $\esssup\calD=0$ a.s. as $\esssup\calD_t=0$ a.s. So, this integral can be taken over $\R$ since the support of $f$ is contained in $[-K_f,\infty)$. We deduce that
\begin{equation}\label{constant-f-D}
\Cb(f)=\Cb_M\int_\R \sqrt{2}\E\left[(1-e^{-\int f(y+z)\mathcal{D}(dz)})\right] e^{-\sqrt{2}y}dy,
\end{equation}
for any $f$ of form \eqref{f-type}.

This allows us to rewrite Proposition \ref{Lap-Et} as follows
\[
\E[e^{-\<{f,\calE_\infty}}]=\E\left[\exp(-Z_\infty \Cb_M\int_\R \sqrt{2}\E\left[(1-e^{-\int f(y+z)\mathcal{D}(dz)})\right] e^{-\sqrt{2}y}dy)\right],
\]
which describes the decoration structure of $\calE_\infty$ in \eqref{def-Et} via the limiting decoration process $\calD$. As a result, the Laplace functional of $\widetilde{\calE}_\infty$ in \eqref{tilde-def} is given by 
\begin{equation}\label{lap-tildeE}
\E\left[e^{-\<{f,\widetilde{\calE}_\infty}}\right]=e^{-\Cb_M\int_\R \sqrt{2}\E\left[(1-e^{-\int f(y+z)\mathcal{D}(dz)})\right] e^{-\sqrt{2}y}dy}=e^{-\Cb(f)},
\end{equation}
for any $f$ of form \eqref{f-type}.

\begin{remark}
Note that one may rephrase our main Theorem  \ref{main-thm} as follows: any fixed point $\theta$ of BBM with critical drift is such that its  Laplace functional must satisfy 
\[
\E[e^{-\<{f,\theta}}]=\E[e^{-e^{\sqrt{2}S} \Cb(f)}],
\]
for any $f$ of the form \eqref{f-type} and for some arbitrary real-valued random variable $S$. 
\end{remark}

\subsection{The $\psi$ function from Bramson.}\label{ss.bramson2}
Our proof will rely at a key place (for the proof of Lemma \ref{bd-L}) on a way to control the solutions of the FKPP equation at a large time $t$ from its behaviour at a much earlier time $r\ll t$. This is quantified using the so-called ``$\psi$-function'' from Bramson \cite{bramson83}. See also the work by Chauvin and Rouault \cite{chauvin} which influenced some of the developments below.
\begin{proposition}[Proposition $8.3$ of \cite{bramson83}, Proposition $4.3$ of \cite{ABK13}]\label{bramson}
Let $\varphi(x)$ be a measurable function satisfying the conditions \eqref{cond1},\eqref{cond2} and \eqref{cond3}. And let $u_\varphi$ be the solution of FKPP equation with initial condition $u(0,x)= \varphi(x)$. Define for any $X\in\R$ and $t>r>0$,
\[\psi(r,t,\sqrt{2}t + X) := \frac{e^{-\sqrt{2}X}}{\sqrt{2\pi(t-r)}}\int_0^\infty u_\varphi(r, y+ \sqrt{2}r)e^{\sqrt{2}y}e^{\frac{-(y-X)^2}{2(t-r)}}\bigl\{1- e^{-2y\frac{X+ \frac{3}{2\sqrt{2}}\log(t)}{t-r}}\bigr\}dy.\]
Then for all $r$ large enough (depending only on the initial condition $u(0,\cdot)$), $t\geq 8r$ and $X \geq 8r - \frac{3}{2\sqrt{2}}\log(t)$, 

\begin{equation}\label{crucial}
\gamma_r^{-1}\psi(r,t,\sqrt{2}t + X) \leq u_\varphi(t, X+ \sqrt{2}t) \leq \gamma_r\psi(r,t,\sqrt{2}t + X),
\end{equation}
where $\gamma_r \downarrow 1$ as $r\to \infty$.
\end{proposition}

\ni
We end this long list of prerequisites with the following two statements on tail estimates describing  essentially what happens away from the window of {\em good} points $[-\sqrt{t}/\delta, -\delta \sqrt{t}]$.

\begin{lemma}[Lemma $4.6$ of \cite{ABK13}]\label{y-order}
Let $\varphi(x)$ be a measurable function satisfying the conditions \eqref{cond1},\eqref{cond2} and \eqref{cond3}. And let $u_\varphi$ be the solution of FKPP equation with initial condition $u(0,x)= \varphi(x)$. Then, for any $x\in\R$,
\[\lim_{A_1 \downarrow 0}\limsup_{r\to \infty}\int_0^{A_1\sqrt{r}}u_\varphi(r, x+y+ \sqrt{2}r)ye^{\sqrt{2}y}dy =0,\]
and 
\[\lim_{A_2 \uparrow \infty}\limsup_{r\to \infty}\int_{A_2\sqrt{r}}^{\infty}u_\varphi(r, x+y+ \sqrt{2}r)ye^{\sqrt{2}y}dy =0.\]
\end{lemma}

\begin{lemma}[Lemma $4.7$ of \cite{ABK13}, Theorem $4.1$ of \cite{M-15}]\label{sharp-bound}
There exist a constant $c_1>0$ such that for each $X\geq -\frac{1}{2}\log(t) $ and $t$ large enough, \[\mathbb{P}(M_t \geq \sqrt{2}t +  X) \leq c_1 (X + \log(t))t^{-\frac{3}{2}}e^{-\sqrt{2}X - \frac{X^2}{2t}}.\]
Moreover, there exists constant $c_2>0$ such that for all $t$ large enough and $-\frac{1}{2}\log(t)\leq X\leq \sqrt{t}$, 
\[
\P(M_t\geq \sqrt{2}t+X) \geq c_2(X+\log t) t^{-3/2}e^{-\sqrt{2}X-\frac{X^2}{2t}}.
\]
\end{lemma}
In fact, Lemma $4.7$ of \cite{ABK13} (or Corollary 10 of \cite{ABK12}) gives the upper bound for the first estimate  while Theorem $4.1$ of \cite{M-15} gives the lower bound for branching random walks in discrete time. However, it is easy to generate this estimate for BBM.

\section{Uniqueness of fixed points} \label{Uniqueness}

This section is devoted to proving the uniqueness part in Theorem \ref{main-thm}. As mentioned in Section \ref{facts-BBM}, we are going to show that if $\theta$ is a fixed point in the space $\M$ (meaning if $\theta\sim \pi$ an invariant probability measure on $\M\subset (\calN,d)$, see Definition \ref{d.FP}), then for any  
 $f$ of form \eqref{f-type} we have
\[
\E[e^{-\<{f,\theta}}]=\E[e^{-Z_\infty^\theta \Cb(f)}],
\]
where $Z_\infty^\theta$ is a positive random variable. In fact, the random shift $S$ stated in Theorem \ref{main-thm} corresponds to  $S=\frac{1}{\sqrt{2}}\log Z_\infty^\theta$.

\begin{remark}
$Z^\theta_\infty$ should not be confused with the derivative martingale $Z_\infty$. We use the same symbol because they play exactly the same role. However, as we will see later, for any positive random variable $Z$, one can construct a fixed point $\theta$ such that that $Z^\theta_\infty = Z$.
\end{remark}

\subsection{Laplace transform at a later time $t>0$.}

If $\theta_0=\theta=\sum_{i\in I}\delta_{x_i}$ is a fixed point a.s. in $\M$, by \eqref{fixedpoint}, one sees that for any $f$ of form \eqref{f-type}, and for any $t>0$,
\begin{align*}
\E[e^{-\<{f, \theta}}]=&\E[e^{-\<{f,\theta_t}}]\\
=&\E\left[e^{-\sum_{i\in I}\sum_{k=1}^{n^i(t)}f(x_i+\chi_k^i(t)-\sqrt{2}t)}\right]
\end{align*}
which by independence is $\E[\prod_{i\in I}(1-u_\varphi(t,\sqrt{2}t-x_i))]$ where $u_\varphi$ is the FKPP solution with initial condition $u(0,x)=\varphi(x)=1-e^{-f(-x)}$. For convenience, let us view $\theta$ as a random measure. Then,
\begin{equation}
\E[e^{-\<{f, \theta_t}}]=\E\left[e^{-\int_{\R}-\log(1-u_\varphi(t,\sqrt{2}t-x))\theta(dx)}\right].
\end{equation}
As we assume that $\theta\in\M$ a.s. we have $\P(\esssup\theta>A)=o_A(1)$ as $A\to\infty$. As a consequence,
\begin{align*}
\E[e^{-\langle f,\theta_t \rangle}]=&\E\left[\prod_{x\in\theta}(1-u(t,\sqrt{2}t-x))\ind{\esssup\theta\leq A}\right]+\underbrace{\E\left[\prod_{x\in\theta}(1-u(t,\sqrt{2}t-x))\ind{\esssup\theta>A}\right]}_{\leq\P(\esssup\theta>A)}\\
=&\E\left[\exp\left\{\int_{-\infty}^A\log(1-u(t,\sqrt{2}t-x))\theta(dx)\right\}\ind{\esssup\theta\leq A}\right]+o_A(1)\\
=&\E\left[\exp\left\{\int_{-\infty}^A\log(1-u(t,\sqrt{2}t-x))\theta(dx)\right\}\right]+o_A(1).
\end{align*}
Next, we show that $u_\varphi(t,\sqrt{2}t-x)=o_t(1)$ uniformly for $x\le A$. In fact, by considering the maximal position $M_t$ of BBM and taking $K_f\ge0$ so that $\textrm{supp}(f)\subset [-K_f,\infty)$,
\begin{align}
u_\varphi(t,\sqrt{2}t-x)=&1-\E\left[e^{-\sum_{k=1}^{n(t)}f(x+\chi_k(t)-\sqrt{2}t)}\right]\nonumber\\
=&\E\left[(1-e^{-\sum_{k\leq n(t)}f(x+\chi_k(t)-\sqrt{2}t)})\ind{x+M_t-\sqrt{2}t\geq-K_f}\right]\nonumber\\
\leq & \P(x+M_t-\sqrt{2}t\geq -K_f).\label{roughbd-u}
\end{align}
Therefore, for any $A>0$,
\[
\sup_{x\le A}u_\varphi(t,\sqrt{2}t-x)\leq \P(M_t\geq \sqrt{2} t-K_f-A)=o_t(1),
\]
as $M_t-m(t)$ converges in law. This ensures that uniformly for $x\leq A$,
\[
\log(1-u_\varphi(t,\sqrt{2}t-x))=-(1+o_t(1))u_\varphi(t,\sqrt{2}t-x).
\]
Consequently,
\begin{equation}\label{before-split}
\E[e^{-\langle f,\theta_t \rangle}]=\E\left[\exp\left\{-(1+o_t(1))\int_{-\infty}^Au_\varphi(t,\sqrt{2}t-x)\theta(dx)\right\}\right]+o_A(1).
\end{equation}
Define
\begin{align*}
\Theta_t(f):=\int_{-\infty}^Au_\varphi(t,\sqrt{2}t-x)\theta(dx).
\end{align*}
\subsection{Dividing initial points into good and bad points.} 
For any $\delta\in(0,1)$ and $t$ sufficiently large, we split $\Theta_t(f)$ into three parts as follows:
\[
\Theta_t(f)= \int_{-\sqrt{t}/\delta}^{-\sqrt{t}\delta}u_\varphi(t,\sqrt{2}t-x)\theta(dx)+ L(t,\delta) + R(t,\delta),
\]
where we introduce the {\em Left} part:
\[L(t,\delta):= \int_{-\infty}^{-\sqrt{t}/\delta}u_\varphi(t,\sqrt{2}t-x)\theta(dx),\]
and the {\em Right} part:
\[R(t,\delta):= \int_{-\delta \sqrt{t}}^Au_\varphi(t,\sqrt{2}t-x)\theta(dx).\]
For {\em good} points $x\in[-\frac{1}{\delta}\sqrt{t},-\delta\sqrt{t}]$, we have
\begin{align*}
u_\varphi(t,\sqrt{2}t-x)=&\E\left[(1-e^{-\sum_{k=1}^{n(t)}f(x+\chi_k(t)-\sqrt{2}t)})\ind{x+M_t-\sqrt{2}t\geq-K_f}\right]\\
=&\E\left[(1-e^{-\int f(x+\overline{M}_t+z)\mathcal{D}_t(dz)})\Big\vert x+\overline{M}_t\geq -K_f \right]\P(x+M_t\geq \sqrt{2}t-K_f),
\end{align*}
where $\mathcal{D}_t=\sum_{1\le k\leq n(t)}\delta_{\chi_k(t)-M_t}$, $\overline{M}_t=M_t-\sqrt{2}t$. Using Lemma \ref{new-lemma} and \eqref{conditionedcvg}, we obtain that uniformly for $x\in[-\frac{1}{\delta}\sqrt{t},-\delta\sqrt{t}]$, one has
\begin{equation*}
u_\varphi(t,\sqrt{2}t-x)=(1+o_t(1))e^{-\sqrt{2}K_f}\frac{\Cb(f)}{\Cb_M}\P(x+M_t\geq \sqrt{2}t-K_f)
\end{equation*}
As a result, \eqref{before-split} becomes 
\begin{equation}\label{Theta-eqn}
    \E[e^{-\langle f,\theta_t \rangle}]=\E[\exp(-(1+o_t(1))\Theta_t(f))]+o_A(1),
\end{equation}
where
\begin{align*}
\Theta_t(f) = &(1+o_t(1))\frac{\Cb(f)}{\Cb_M}\int_{-\frac{1}{\delta}\sqrt{t}}^{ -\delta\sqrt{t}}e^{-\sqrt{2}K_f}\P(x+M_t\geq \sqrt{2}t-K_f) \theta(dx)+L(t,\delta)+R(t,\delta).
\end{align*}
Define
\[Z_t^\theta(K_f,\delta) := \frac{1}{\Cb_M}\int_{-\frac{1}{\delta}\sqrt{t}}^{ -\delta\sqrt{t}}e^{-\sqrt{2}K_f}\P(x+M_t\geq \sqrt{2}t-K_f)\theta(dx)\]
One then gets
\[
\Theta_t(f)=(1+o_t(1)) Z_t^\theta(K_f,\delta)\Cb(f)+L(t,\delta)+R(t,\delta).
\]
On the one hand, observe that by \eqref{roughbd-u}, 
\begin{align*}
0\leq L(t,\delta)\leq &\int_{-\infty}^{-\sqrt{t}/\delta}\P(x+M_t\geq \sqrt{2}t-K_f)\theta(dx)=: L^+(t,\delta),\\ 
0\leq R(t,\delta)\leq &\int_{-\delta\sqrt{t}}^A\P(x+M_t\geq \sqrt{2}t-K_f)\theta(dx)=: R^+(t,\delta).
\end{align*}
On the other hand, recall that $\P(x+M_t\geq \sqrt{2}t-K_f)=u_M(t,\sqrt{2}t-x-K_f)$. By Lemma \ref{decoration-lemma},  uniformly for $x\in [-\frac{1}{\delta}\sqrt{t},-\delta\sqrt{t}]$, 
\[
\P(x+M_t\geq \sqrt{2}t-K_f)= (1+o_t(1))\Cb_M t^{-3/2}|x|e^{\sqrt{2}(x+K_f)-\frac{x^2}{2t}}.
\]
So,
\begin{align*}
Z_t^\theta(K_f,\delta)
=&(1+o_t(1))\int_{-\frac{1}{\delta}\sqrt{t}}^{ -\delta\sqrt{t}}t^{-3/2}|x|e^{\sqrt{2}x-\frac{x^2}{2t}}\theta(dx)\\
=&\frac{1}{\Cb_M}(1+o_t(1))\int_{-\frac{1}{\delta}\sqrt{t}}^{ -\delta\sqrt{t}}\P(x+M_t\geq \sqrt{2}t)\theta(dx)=(1+o_t(1))Z_t^\theta(0,\delta).
\end{align*}
The above identity is important as it shows that the {\em middle} part $Z_t^\theta(K_f,\delta)$ in fact does not depend much on the width of the support of $f$, i.e. $K_f$. This leads us to introduce the following quantity which does not depend neither on $K_f$ nor on $\delta$.
Define 
\begin{equation}
Z_t^\theta:=\frac{1}{\Cb_M}\int_{-\infty}^0 \P(x+M_t\geq \sqrt{2}t)\theta(dx).
\end{equation}
We then rewrite $\Theta_t(f)$ as follows:
\begin{equation}
\Theta_t(f)=(1+o_t(1))\Cb(f) Z_t^\theta + L_1(t,\delta)+R_1(t,\delta)
\end{equation}
where 
\[
L_1(t,\delta)=L(t,\delta)-(1+o_t(1))\frac{\Cb(f)}{\Cb_M}\int_{-\infty}^{-\sqrt{t}/\delta}\P(x+M_t\ge\sqrt{2}t)\theta(dx),
\]
and 
\[
R_1(t,\delta)=R(t,\delta)-(1+o_t(1))\frac{\Cb(f)}{\Cb_M}\int_{-\sqrt{t}\delta}^0\P(x+M_t\ge\sqrt{2}t)\theta(dx).
\]
Observe immediately that there exists some constant $c_3>1$ such that for sufficiently large $t$,
\[
|L_1(t,\delta)|\leq c_3 L^+(t,\delta)\textrm{ and } |R_1(t,\delta)|\le c_3 R^+(t,\delta).
\]
Our uniqueness of fixed points (Theorem \ref{main-thm}), as we shall explain in the next Section follows from the following key Lemma.
\begin{lemma}\label{l.key}
The random variables $Z_t^\theta, R^+(t,\delta)$ and $L^+(t,\delta)$ which are each measurable w.r.t to the initial point process $\theta$ satisfy 
\begin{enumerate}
\item $(Z_t^\theta)_{t>0}$ is tight;
\item $R^+(t,\delta)$ converges in probability to 0 as $t\to\infty$ and then $\delta\to0$;
\item $L^+(t,\delta)$ converges in probability to 0 as $t\to\infty$ and then $\delta\to0$.
\end{enumerate}
\end{lemma}

\subsection{Proof of Theorem \ref{main-thm} given Lemma \ref{l.key}.}
$ $

\ni
{\em Proof of uniqueness in Theorem \ref{main-thm}.}
Assuming Lemma \ref{l.key}, we first obtain the tightness $(\Theta_t(f))_{t>0}$.  It also follows that along some subsequence of $t$, $Z_t^\theta$ converges in law to some non-negative random variable $Z_\infty^\theta$. Then along the same subsequence of $t$ and letting first $\delta\downarrow0$ and then $A\to\infty$, we obtain that
\[
\E[e^{-\<{f,\theta_t}}]\to \E[e^{-Z_\infty^\theta\Cb(f)}].
\]
Therefore $\E[e^{-\<{f,\theta}}]=\E[e^{-Z_\infty^\theta\Cb(f)}]$. Moreover, we get that $Z_\infty^\theta>0$ a.s. because we assumed $\theta\in \M \setminus\{0\}$ a.s.  We thus conclude the uniqueness of fixed points.

\begin{remark}
In fact, for a fixed point $\theta$, $Z_t^\theta$ converges in law to some positive random variable $Z_\infty^\theta$ so that $\theta$ is distributed as $\widetilde{\calE}_\infty$ shifted by $S=\frac{1}{\sqrt{2}}\log Z_\infty^\theta$.
\end{remark}

We prove the three assertions in the following subsections. Note that the tightness of $(Z_t^\theta)$ follows from the tightness of $(Z_t^\theta(0,\delta))_{t>0}$ for any fixed $\delta\in(0,1)$ and the assertions (2)-(3).

As explained in Section \ref{proof-sketch}, our proof relies on a concentration inequality for sums of independent Bernoulli random variables, stated in the following Lemma.

\begin{lemma}\label{concentration}
Let $\{X_i\}_{i\in\N^*}$ be independent Bernoulli random variables such that $\E[X_i]=p_i$. For any subset $I\subset \N^*$, let $X_I:=\sum_{i\in I}X_i$. If $\E[X_I]=\sum_{i\in I}p_i$ is finite, then for any $\varepsilon\in(0,1)$,
\[
\P\left(|X_I-\E[X_I]|\ge \varepsilon \E[X_I]\right)\leq \frac{1}{\varepsilon^2\E[X_I]}.
\]
If $\E[X_I]=\sum_{i\in I}p_i=\infty$, then $I$ is an infinite set and 
\[
X_I=+\infty,\textrm{ a.s. }
\]
\end{lemma}

If $\E[X_I]<\infty$, this Lemma follows directly from Chebyshev's inequality by noting that $\textrm{Var}(X)=\sum_{i=1}^n p_i(1-p_i)\leq \E[X]$. If $\E[X_I]=\infty$, the result is nothing but Borel-Cantelli Lemma. It indicates that if $\E[X]$ is large, then with high probability, $X$ is also large as it is comparable with its mean.

\subsection{Tightness of $(Z_t^\theta(0,\delta))_{t>0}$ for any fixed $\delta\in(0,1)$.}

Recall that $\theta=\sum_{i\in I}\delta_{x_i}$. Define 
\[
\mathcal{Z}_t^\theta(0,\delta):=\sum_{i: x_i\in[-\sqrt{t}/\delta, -\delta\sqrt{t}]}\ind{x_i+M^i_t\ge\sqrt{2}t}.
\]
Conditionally on $\theta$, $\mathcal{Z}_t^\theta(0,\delta)$ is sum of independent Bernoulli random variables such that $\E[\mathcal{Z}_t^\theta(0,\delta)\vert \theta]=Z_t^\theta(0,\delta)$. Note that $Z_t^\theta(0,\delta)\le \theta([-\sqrt{t}/\delta, -\delta\sqrt{t}])<\infty$ a.s. It thus follows that for any $K>0$,
\begin{align}
&\P(Z_t^\theta(0,\delta)\geq K)\nonumber\\
\leq & \P\left(Z_t^\theta(0,\delta)\geq K; |\mathcal{Z}_t^\theta(0,\delta)-Z_t^\theta(0,\delta)|\geq \frac12 Z_t^\theta(0,\delta)\right)+\P\left(Z_t^\theta(0,\delta)\geq K; \mathcal{Z}_t^\theta(0,\delta)\ge\frac12 Z_t^\theta(0,\delta)\right)\nonumber\\
\leq & \E\left[\ind{Z_t^\theta(0,\delta)\geq K}\P(|\mathcal{Z}_t^\theta(0,\delta)-Z_t^\theta(0,\delta)|\geq \frac12 Z_t^\theta(0,\delta)\vert\theta)\right]+\P\left(\mathcal{Z}_t^\theta(0,\delta)\geq \frac12 K\right)\nonumber\\
\leq & \E\left[\frac{4}{Z_t^\theta(0,\delta)}\ind{Z_t^\theta(0,\delta)\geq K}\right]+\P\left(\mathcal{Z}_t^\theta(0,\delta)\geq \frac12 K\right),\label{Keycompare}
\end{align}
where the last inequality comes from Lemma \ref{concentration}. We hence get that for any $K>0$ and $t>0$,
\[
\P(Z_t^\theta(0,\delta)\geq K)\leq \frac{4}{K}+\P\left(\mathcal{Z}_t^\theta(0,\delta)\geq \frac12 K\right).
\]
From the definition of $\mathcal{Z}_t^\theta(0,\delta)$, one sees easily that
\[
\mathcal{Z}_t^\theta(0,\delta) \le \theta_t(\R_+), \forall t>0,
\]
where $\theta_t(\R_+)$ is distributed as $\theta(\R_+)$. As assumed at the beginning, $\theta(\R_+)<\infty$ a.s. Therefore, one has
\[
\sup_{t>0}\P(Z_t^\theta(0,\delta)\geq K)\leq \frac{4}{K}+\P\left(\theta(\R_+)\ge \frac12 K\right)\to 0, \textrm{ as } K\to\infty.
\]
It suffices to conclude the tightness of $(Z_t^\theta(0,\delta))_{t>0}$.

\subsection{Convergence in probability of $R^+(t,\delta)$.}

Let us first consider
\[
R^+(t,\delta)=\int_{-\delta\sqrt{t}}^A\P(x+M_t\geq \sqrt{2}t-K_f)\theta(dx),
\]
which is easier to deal with than $L^+(t,\delta)$.

Our goal is to show that for any $\eta>0$,
\[
\limsup_{\delta\to0}\limsup_{t\to\infty}\P(R^+(t,\delta)\geq \eta)=0.
\]
By contradiction, if it fails,  there exist $\eta>0$ and $\varepsilon>0$ such that along some subsequence $(t_k,\delta_k)$ with $\delta_k\downarrow 0$ and $t_k\ge \delta_k^{-5}$,
\begin{equation}\label{hyp-R}
\P(R^+(t,\delta)\geq \eta)\geq \varepsilon>0.
\end{equation}
Similarly as the previous subsection, with a little more generality, we define for any $t,s>0$,
\[
\mathcal{R}^+_s(t,\delta):=\sum_{i: x_i\in [-\delta\sqrt{t},A]}\ind{x_i+M^i_s-\sqrt{2} s \ge -K_f}.
\]
In parallel with it, we define
\[
R^+_s(t,\delta):=\int_{-\delta\sqrt{t}}^A \P(x+M_s-\sqrt{2} s\ge -K_f)\theta(dx),
\]
with $R^+_t(t,\delta)=R^+(t,\delta)$. Apparently, $\mathcal{R}^+_s(t,\delta)\le \theta_s([-K_f,\infty))<\infty$ a.s. and $\E[\mathcal{R}^+_s(t,\delta)\vert \theta]=R_s^+(t,\delta)$. To apply the same idea as above, one needs to find some $s>0$ so that $R^+_s(t,\delta)$ could be very large. This is guaranteed by the following lemma.

\begin{lemma}\label{ratio-R}
Let $\delta\in(0,1)$ and $t\ge \delta^{-5}$. Take $s=\delta^2t$. Then, for $\delta$ sufficiently small ($t$ sufficiently large), there exists some constant $c_4>0$ such that
\[
\inf_{x\in[-\delta\sqrt{t}, A]}\frac{\P(x+M_s-\sqrt{2} s\ge -K_f)}{\P(x+M_t-\sqrt{2} t\ge -K_f)}\ge c_4 \delta^{-3}.
\]
\end{lemma}
This lemma follows from Lemma \ref{sharp-bound}. Its proof is postponed to the end of this section.

Because of Lemma \ref{ratio-R}, one sees that for $s=\delta^2 t$,
\begin{equation*}
 \P(R^+(t,\delta)\geq \eta) \le \P(R^+_s(t,\delta)\ge c_4\eta \delta^{-3})
\end{equation*}
Redo the same estimate for $R^+_s(t,\delta)$ and $\mathcal{R}_s^+(t,\delta)$ as in \eqref{Keycompare}, we get that
\[
 \P(R^+_s(t,\delta)\ge c_4\eta \delta^{-3})\leq \frac{4}{c_4\eta\delta^{-3}}+\P(\theta([-K_f,\infty)\ge \frac{c_4}{2}\eta \delta^{-3})
\]
Letting $\delta\downarrow 0$ along the subsequence $(\delta_k)$ leads to a contradiction with \eqref{hyp-R}.

\subsection{Convergence in probability of $L^+(t,\delta)$.} Let us now turn to the study of the more delicate term
\[
L^+(t,\delta)=\int_{-\infty}^{-\sqrt{t}/\delta}\P(x+M_t\ge \sqrt{2}t-K_f)\theta(dx).
\]
In order to show that for any $\eta>0$,
\[
\limsup_{\delta\to0}\limsup_{t\to\infty}\P(L^+(t,\delta)\geq \eta)=0,
\]
we suppose by contradiction that there exist $\eta>0$ and $\varepsilon>0$ such that along some subsequence $(t_k,\delta_k)$ with $\delta_k\downarrow 0$ and $t_k\uparrow\infty$,
\begin{equation}\label{hyp-L}
\P(L^+(t,\delta)\geq \eta)\geq \varepsilon>0.
\end{equation}
In the same spirits as above, we define for any $s,t>0$,
\[
\mathcal{L}_s^+(t,\delta):=\sum_{i: x_i\le -\sqrt{t}/\delta}\ind{x_i+M^i_t-\sqrt{2}t\ge -K_f},
\]
as well as
\begin{equation*}
L^+_s(t,\delta)=\int_{-\infty}^{-\sqrt{t}/\delta}\P(x+M_s\ge \sqrt{2}s-K_f)\theta(dx),
\end{equation*}
with $L^+_t(t,\delta)=L^+(t,\delta)$.

One would expect to establish a similar result as Lemma \ref{ratio-R} for some well chosen $s$. Nevertheless, things will be significantly more delicate here as we do not have sharp uniform bounds for $\P(M_t\ge \sqrt{2}t-x-K_f)$ with any $x\in(-\infty, -\sqrt{t}/\delta]$. In the following lemma, when $s=2t$, we manage to compare $L^+_s(t,\delta)$ and $L^+(t,\delta)$ by analysing $\P(M_t\ge \sqrt{2}t-x-K_f)$ via $\psi(r,t,\sqrt{2}t+X)$ introduced in Proposition \ref{bramson}.

\begin{lemma}\label{bd-L}
Take $s=2t$. For any $\delta\in(0,1)$ sufficiently small, there exists $B_{\delta}\ge1$ such that for any $\eta >0$, as $t\to\infty$,
\[
\P\left(L^+(t,\delta))\ge \eta\right)\leq \P(L_s^+(t,\delta)\geq B_{\delta}\eta ) + o_t(1).
\]
Moreover, $B_\delta\uparrow\infty$ as $\delta\downarrow0$.
\end{lemma} 

Its proof is postponed to the end of this section.

Again, like \eqref{Keycompare}, by Lemma \ref{concentration}, along the subsequence $(\delta_k, t_k)$, one has
\begin{align*}
\P(L_s^+(t,\delta)\geq B_{\delta}\eta )=&\P(B_{\delta}\eta \le L_s^+(t,\delta)<\infty )+\P(L_s^+(t,\delta)=+\infty)\\
\leq &\frac{4}{B_\delta \eta}+\P( \frac12 B_\delta\eta\le \theta([-K_f,\infty)<\infty)+\P(\theta([-K_f,\infty)=+\infty)\\
= &\frac{4}{B_\delta \eta}+o_\delta(1).
\end{align*}
So, by Lemma \ref{bd-L}, 
\[
\P(L^+(t,\delta)\ge \eta)\le \frac{4}{B_\delta \eta}+o_\delta(1)+o_t(1).
\]
Letting $t\to\infty$ and $\delta\to0$ along the subsequence $(t_k,\delta_k)$ leads to a contradiction with \eqref{hyp-L}.

In the next subsection, we prove Lemmas \ref{ratio-R} and \ref{bd-L}.

\subsection{Proofs of Lemmas \ref{ratio-R} and \ref{bd-L}.} We first prove Lemma \ref{ratio-R}.

\begin{proof}
For $t$ sufficiently large and $x\in[-\delta\sqrt{t},A]$, by Lemma \ref{sharp-bound} with $X=-x-K_f$,
\[
\P(x+M_t-\sqrt{2}t\ge -K_f)\le c_1(\log t-x-K_f)t^{-3/2}e^{\sqrt{2}(x+K_f)}.
\]
Again, applying Lemma \ref{sharp-bound} for $s=\delta^2 t$ with $t$ sufficiently large implies that
\[
\P(x+M_s-\sqrt{2}s\ge-K_f)\ge c_2(\log s-x-K_f) s^{-3/2}e^{\sqrt{2}(x+K_f)-\frac{(x+K_f)^2}{2s}}.
\]
It then follows that
\[
\frac{\P(x+M_s-\sqrt{2}s\ge-K_f)}{\P(x+M_t-\sqrt{2}t\ge -K_f)}\ge \frac{c_2\delta^{-3}}{c_1}\frac{\log s-x-K_f}{\log t -x-K_f} e^{-\frac{(x+K_f)^2}{2s}}.
\]
As $t\ge \delta^{-5}$ with $\delta>0$ small, one has
\[
\frac{\log s-x-K_f}{\log t -x-K_f}\ge 1/2, \forall x\in[-\delta\sqrt{t},A].
\]
Besides, 
\[
\sup_{x\in[-\delta\sqrt{t},A]}\frac{(x+K_f)^2}{2s}\le \frac{(-\delta\sqrt{t}+K_f)^2}{2\delta^2t}\le 1/2.
\]
We hence conclude Lemma \ref{ratio-R}.
\end{proof}

To prove Lemma \ref{bd-L}, we are going to analyse $\P(x+M_t-\sqrt{2}t\ge -K_f)$ by use of Bramson's ``$\psi$-function'' defined in Proposition \ref{bramson}. 
\begin{proof}
Note that
\[
\P(x+M_s-\sqrt{2}s\ge -K_f)=u_M(s, \sqrt{2}s-x-K_f), \forall x\in\R, \forall s>0,
\]
and that
\begin{equation}\label{def-Ls}
L^+_s(t,\delta)=\int_{-\infty}^{-\sqrt{t}/\delta}u_M(s,\sqrt{2}s-x-K_f)\theta(dx).
\end{equation}
Obviously, the corresponding initial condition $u_M(0,\cdot)$ satisfies \eqref{cond1}, \eqref{cond2} and \eqref{cond3}. So, Propositions \ref{bramson} implies that for $r\le \frac{\sqrt{t}}{8\delta}$ with $r$ large enough,
\[
 u_M(t,\sqrt{2}t-x-K_f) \leq \gamma_r\psi(r,t,\sqrt{2}t-x-K_f), \forall x\le -\sqrt{t}/\delta,
\]
where $\gamma_r\downarrow 1$ as $r\uparrow\infty$. Meanwhile, for $s=2t$, we also have
\[
u_M(s,\sqrt{2}s-x-K_f)\ge \gamma_r^{-1} \psi(r, s, \sqrt{2}s-x-K_f), \forall x\le -\sqrt{t}/\delta.
\]
Recall that
\begin{multline}
\psi(r,t,\sqrt{2}t-x-K_f)\\
=\frac{e^{\sqrt{2}(x+K_f)}}{\sqrt{2\pi(t-r)}}\int_0^\infty u_M(r,\sqrt{2}r+y)e^{\sqrt{2}y}e^{-\frac{(y+x+K_f)^2}{2(t-r)}}(1-e^{-2y\frac{-x-K_f+\frac{3}{2\sqrt{2}}\log t}{t-r}})dy.
\end{multline}
Notice that for $x\le -\sqrt{t}/\delta$ with sufficiently large $t$,
\[
1- e^{-2y\frac{\frac{3}{2\sqrt{2}}\log(t)- x- K_f}{t-r}} \leq g(y), \forall y\in\R_+,
\]
where
\begin{equation}\label{def-g}
g(y):=  2y\frac{\frac{3}{2\sqrt{2}}\log(t) - x- K_f}{t-r}\ind{y\leq \frac{t-r}{\frac{3}{2\sqrt{2}}\log(t) - x- K_f} }+\ind{y \geq \frac{t-r}{\frac{3}{2\sqrt{2}}\log(t) - x- K_f}}.
\end{equation}
On the other hand, for $s=2t$ and $\sqrt{t}/8\ge r\gg 1$, 
\[
1- e^{-2y\frac{\frac{3}{2\sqrt{2}}\log(s) - x- K_f}{s-r}} \geq 1- e^{-2y\frac{\frac{3}{2\sqrt{2}}\log(t) - x- K_f}{3(t-r)}} \geq c_5 g(y),
\]
with some constant $c_5>0$. As a consequence, for $x\le -\sqrt{t}/\delta$ with $t\gg1$,
\begin{align*}
\psi(r,t,\sqrt{2}t-x-K_f) \le &\frac{e^{\sqrt{2}(x+K_f)}}{\sqrt{2\pi(t-r)}}\int_0^\infty u_M(r,\sqrt{2}r+y)e^{\sqrt{2}y}e^{-\frac{(y+x+K_f)^2}{2(t-r)}}g(y)dy,\\
\psi(r,s,\sqrt{2}s-x-K_f) \ge& c_5\frac{e^{\sqrt{2}(x+K_f)}}{\sqrt{6\pi(t-r)}}\int_0^\infty u_M(r,\sqrt{2}r+y)e^{\sqrt{2}y}e^{-\frac{(y+x+K_f)^2}{2(s-r)}}g(y)dy.
\end{align*}
It remains to compare $e^{-\frac{(y+x+K_f)^2}{2(t-r)}}$ with $e^{-\frac{(y+x+K_f)^2}{2(s-r)}}$. Observe that for any $x,y\in\R$,
\[ 
\frac{-(y+ x+ K_f)^2}{2(s-r)} \geq \frac{-(y+ x+ K_f)^2}{2(t-r)} + \frac{(y+ x+ K_f)^2}{4(t-r)}. 
\]
Therefore,
\[
e^{-\frac{(y+x+K_f)^2}{2(s-r)}}\ge e^{-\frac{(y+x+K_f)^2}{2(t-r)}}\times\left(e^{\frac{1}{36\delta^2}}\ind{|y+x+K_f|\ge \frac{\sqrt{t}}{3\delta}}+\ind{|y+x+K_f|<\frac{\sqrt{t}}{3\delta}}\right).
\]
So, we obtain that
\begin{align*}
 u_M(t,\sqrt{2}t-x-K_f) \leq &\gamma_r\left(\psi_\ge(r,t,x)+\psi_\le(r,t,x)\right),\\
 \textrm{ and }u_M(s,\sqrt{2}s-x-K_f)\ge &c_6\gamma_r^{-1}\left(e^{\frac{1}{36\delta^2}}\psi_\ge(r,t,x)+\psi_\le(r,t,x)\right),
\end{align*}
where $c_6=\frac{c_5}{\sqrt{3}}$ and
\begin{align*}
\psi_\ge(r,t,x):= &\frac{e^{\sqrt{2}(x+K_f)}}{\sqrt{2\pi(t-r)}}\int_0^\infty u_M(r, y+ \sqrt{2}r)e^{\sqrt{2}y}e^{\frac{-(y+x+ K_f)^2}{2(t-r)}}g(y)\ind{|y+ x+ K_f| \geq \frac{\sqrt{t}}{3\delta}}dy, \\
\psi_\le(r,t,x):= &\frac{e^{\sqrt{2}(x+K_f)}}{\sqrt{2\pi(t-r)}}\int_0^\infty u_M(r, y+ \sqrt{2}r)e^{\sqrt{2}y}e^{\frac{-(y+x+ K_f)^2}{2(t-r)}}g(y)\ind{|y+ x+ K_f| <  \frac{\sqrt{t}}{3\delta}}dy.
\end{align*}
Going back to \eqref{def-Ls}, we see that
\begin{align*}
L^+(t,\delta)\le & \gamma_r\int_{-\infty}^{-\sqrt{t}/\delta}\psi_\ge(r,t,x)\theta(dx)+\gamma_r\int_{-\infty}^{-\sqrt{t}/\delta}\psi_\le(r,t,x)\theta(dx),\\
\textrm{ and } L_s^+(t,\delta)\ge& c_6\gamma_r^{-1} e^{\frac{1}{36\delta^2}}\int_{-\infty}^{-\sqrt{t}/\delta}\psi_\ge(r,t,x)\theta(dx)+c_6\gamma_r^{-1}\int_{-\infty}^{-\sqrt{t}/\delta}\psi_\le(r,t,x)\theta(dx).
\end{align*}
Recall that for $\delta\in(0,1)$, we take $r\le \frac{\sqrt{t}}{8\delta}$. We now claim that for $\delta\in(0,1/8)$ and $r=\sqrt{t}$, as $t\to \infty$,
\begin{equation}\label{small-term-psi}
\int_{-\infty}^{-\sqrt{t}/\delta}\psi_\le(r,t,x)\theta(dx) \overset{\P}{\longrightarrow} 0.    
\end{equation}
The proof of above claim is deferred to the end. By admitting this claim, we are ready to stress our argument: for any $\eta>0$ and $t$ sufficiently large,
\begin{align*}
&\P(L^+(t,\delta)>\eta)\\
\leq & \P\left(\gamma_r\int_{-\infty}^{-\sqrt{t}/\delta}\psi_\ge(r,t,x)\theta(dx)\ge \eta/2\right)+\P\left(\gamma_r\int_{-\infty}^{-\sqrt{t}/\delta}\psi_\le(r,t,x)\theta(dx)\ge \eta/2\right)\\
\leq &\P(L_s(t,\delta)\geq B_\delta \eta)+o_t(1),
\end{align*}
where $B_\delta:=\frac{c_6}{2}\gamma^{-2}_r e^{\frac{1}{36\delta^{2}}}$ goes to infinity as $\delta\to0$.
\end{proof}
It remains to check the claim \eqref{small-term-psi}. Recall the definition \eqref{def-g} of $g$, one sees that for $x\leq -\sqrt{t}/\delta$ with $\delta\in (0, 1/8)$ and $t$ sufficiently large,
\[
g(y)\ind{|y+ x+ K_f| <  \frac{\sqrt{t}}{3\delta}}=\ind{|y+ x+ K_f| <  \frac{\sqrt{t}}{3\delta}}.
\]
So,
\[
\psi_\le(r,t,x)= \frac{e^{\sqrt{2}(x+K_f)}}{\sqrt{2\pi(t-r)}}\int_0^\infty u_M(r, y+ \sqrt{2}r)e^{\sqrt{2}y}e^{\frac{-(y+x+ K_f)^2}{2(t-r)}}\ind{|y+ x+ K_f| \leq  \frac{\sqrt{t}}{3\delta}}dy.
\]
Next, by the upper bound in Lemma \ref{sharp-bound}, for any $y>0$,
\[
u_M(r, \sqrt{2}r + y)=\P(M_r\ge \sqrt{2}r+y) \leq c_1(y+ \log(r))r^{-\frac{3}{2}}e^{-\sqrt{2}y- \frac{y^2}{2r}}.
\]
Recall that we have chosen $r= \sqrt{t}$. It is easy to check that for $x \leq -\sqrt{t}/\delta$ and $|y+ x+ K_f| \leq \frac{\sqrt{t}}{3\delta}$, 
\[
u_M(r, \sqrt{2}r + y)e^{\sqrt{2}y}\le c_1(y+ \log(r))r^{-\frac{3}{2}}e^{- \frac{y^2}{2r}} \leq  c_1e^{-\frac{x^2}{16r}},
\]
for $t$ sufficiently large. 

Therefore,
\begin{align*}
\psi_\le(r,t,x)
&\leq c_7e^{\sqrt{2}x - \frac{x^2}{16r}}\frac{1}{\sqrt{2\pi(t-r)}}\int_0^\infty e^{\frac{-(y+x+ K_f)^2}{2(t-r)}}\ind{|y+ x+ K_f| \leq  \frac{\sqrt{t}}{3\delta}}dy \\
& \leq c_7 e^{\sqrt{2}x - \frac{x^2}{16\sqrt{t}}},
\end{align*}
with $c_7=c_1e^{\sqrt{2}K_f}$.  So, to conclude \eqref{small-term-psi}, it suffices to show the following lemma.

\begin{lemma}\label{small-term}
If $\theta \in \M$ is a fixed point of BBM, then as $t\to \infty$, 
\begin{equation}\label{aux-claim}
\int_{-\infty}^{\frac{-\sqrt{t}}{\delta}}e^{\sqrt{2}x - \frac{x^2}{16\sqrt{t}}} \theta(dx) \overset{\P} \longrightarrow 0.
\end{equation}

\end{lemma}

\begin{proof} Let
\[
\mathcal{Y}_t:=\sum_{i: x_i\in (-\infty,-\sqrt{t}/\delta]}\ind{x_i+M^i(s_t)-\sqrt{2}s_t\ge0}.
\]
with $s_t=9\sqrt{t}$. Then, we have
\[
\E[\mathcal{Y}_t\vert\theta]=\int_{-\infty}^{\frac{-\sqrt{t}}{\delta}}\P(x+M_{s_t}-\sqrt{2}s_t\ge0) \theta(dx) \textrm{ and } \mathcal{Y}_t\le \theta_t(\R_+)<\infty \textrm{ a.s.}
\]
In fact, we could compare $\E[\mathcal{Y}_t\vert\theta]$ with $\int_{-\infty}^{\frac{-\sqrt{t}}{\delta}}e^{\sqrt{2}x - \frac{x^2}{16\sqrt{t}}} \theta(dx)$ by checking the following convergence: for any fixed $\delta\in(0,1/8)$,
\begin{equation}\label{aux-1}
\lim_{t\to \infty}\inf_{x\leq \frac{-\sqrt{t}}{\delta}}\frac{\P(x+ M_{s_t} \geq \sqrt{2}s_t)}{e^{\sqrt{2}x - \frac{x^2}{16\sqrt{t}}}} = \infty.
\end{equation}
Recall that we have chosen $r=\sqrt{t}$. Again by Proposition \ref{bramson}, 
\begin{equation}\label{psi-again}
    \P(x+ M_{s_t} \geq \sqrt{2}s_t)= u_M(s_t, \sqrt{2}s_t-x) \geq \gamma_{r}^{-1}\psi(\sqrt{t}, s_t, \sqrt{2}s_t-x).
\end{equation} 
It implies that
\begin{align*}
&\frac{\P(x+ M_{s_t} \geq \sqrt{2}s_t)}{e^{\sqrt{2}x - \frac{x^2}{16\sqrt{t}}}} \\
&\ge\gamma_r^{-1} \sqrt{\frac{1}{16\pi}}\int_0^\infty u_M(\sqrt{t}, y+ \sqrt{2}\sqrt{t})ye^{\sqrt{2}y}e^{\frac{x^2-(y+x)^2}{16\sqrt{t}}}\frac{1}{y t^{1/4}}\bigl\{1- e^{-2y\frac{-x + \frac{3}{2\sqrt{2}}\log(9\sqrt{t})}{8\sqrt{t}}}\bigr\}dy.
\end{align*}
Recall that $\Cb_M=\lim_{t\to\infty}\sqrt{\frac{2}{\pi}}\int_0^\infty u_M(\sqrt{t}, y+ \sqrt{2}\sqrt{t})ye^{\sqrt{2}y}$. In view of Lemma \ref{y-order}, we can find constants $0< A_1 < A_2 < \infty$ such that for all $t$ large enough, 
\[ \sqrt{\frac{2}{\pi}}\int_{A_1t^{\frac{1}{4}}}^{A_2 t^{\frac{1}{4}}}u_M(\sqrt{t}, y+ \sqrt{2}\sqrt{t})ye^{\sqrt{2}y} dy \geq \frac{\Cb_M}{2}.\]
Consequently, one has
\begin{multline*}
\inf_{x\leq \frac{-\sqrt{t}}{\delta}}\frac{\P(x+ M_{s_t} \geq \sqrt{2}s_t)}{e^{\sqrt{2}x - \frac{x^2}{16\sqrt{t}}}} \\
\ge \frac{\Cb_M\gamma_r^{-1}}{8\sqrt{2}} \inf_{y\in [A_1t^{\frac{1}{4}},A_2t^{\frac{1}{4}}], x\leq -\sqrt{t}/\delta}e^{\frac{x^2-(y+x)^2}{16\sqrt{t}}}\frac{1}{y t^{1/4}}\bigl\{1- e^{-2y\frac{-x + \frac{3}{2\sqrt{2}}\log(9\sqrt{t})}{8\sqrt{t}}}\bigr\},
\end{multline*}
where it is easy to check that uniformly in $ y\in [A_1t^{\frac{1}{4}},A_2t^{\frac{1}{4}}]$ and $x\leq -\sqrt{t}/\delta$, 
\[ e^{\frac{x^2-(y+x)^2}{16\sqrt{t}}}\frac{1}{yt^{1/4}}\bigl\{1- e^{-2y\frac{-x + \frac{3}{2\sqrt{2}}\log(9\sqrt{t})}{8\sqrt{t}}}\bigr\} \to \infty \hspace{2mm}\mbox{as}\hspace{2mm} t\to \infty.\]
We hence obtain \eqref{aux-1}.

Let $\alpha_t:=\inf_{x\leq \frac{-\sqrt{t}}{\delta}}\frac{\P(x+ M_{s_t} \geq \sqrt{2}s_t)}{e^{\sqrt{2}x - \frac{x^2}{16\sqrt{t}}}}$. Then, for any $\eta >0$,
\begin{align*}
\mathbb{P}\left(\int_{-\infty}^{\frac{-\sqrt{t}}{\delta}}e^{\sqrt{2}x - \frac{x^2}{16\sqrt{t}}} \theta(dx) > \eta \right) &\leq \mathbb{P}\left(\E[\mathcal{Y}_t\vert \theta] > \alpha_t \eta \right).
\end{align*}
Again, like \eqref{Keycompare}, applying Lemma \ref{concentration} to $\mathcal{Y}_t$, we get that
\[
\mathbb{P}\left(\E[\mathcal{Y}_t\vert \theta] > \alpha_t \eta \right)\le \frac{4}{\alpha_t \eta}+\P(\theta(\R_+)\ge \alpha_t \eta/2).
\]
Letting $t\to\infty$ implies that for any $\eta>0$,
\[
\lim_{t\to\infty}\mathbb{P}\left(\int_{-\infty}^{\frac{-\sqrt{t}}{\delta}}e^{\sqrt{2}x - \frac{x^2}{16\sqrt{t}}} \theta(dx) > \eta \right) =0\,, 
\]
which thus ends our proof. 
\end{proof}

\section{A large subspace left invariant under BBM.}\label{s.basin}


In this section, we introduce a space $\M_{3/2}$ which as opposed to $\M$ has the property that any point process $\theta\in \M_{3/2}$ will satisfy that the BBM (with critical drift or not) $\theta_t$ will still belong to $\M_{3/2}$ a.s. We believe this space may be a natural candidate for the analysis of the basin of attraction both for Liggett's fixed points from \cite{liggett78} as well as for the fixed points from Theorem \ref{main-thm}. 

\begin{definition}\label{d.M32}
We define the following sub-space of $\M$
\begin{equation}\label{non-explosive}
\M_{3/2}:=\{\eta \in\M: \exists \beta\in(0,1)\textrm{ s.t. } \int_{-\infty}^{+\infty} e^{- \beta |x|^{3/2}}\mu(dx) < \infty \}\,.
\end{equation}
\end{definition}
\begin{remark}\label{}
One may also wish to consider the more general spaces $\M_\alpha$ for any exponent $1<\alpha<2$ but for simplicity we stick to $\alpha=3/2$ here. 
\end{remark}
\begin{remark}\label{}
Notice the {\em open condition} $\exists \beta\in (0,1)$. Without this condition, one can prove (by designing specific counter-examples) that the slightly simpler spaces
\[
\M_{3/2,\beta}:= \{\eta \in\M: \textrm{ s.t. } \int_{-\infty}^{+\infty} e^{- \beta |x|^{3/2}}\mu(dx) < \infty \}\,,
\]
are not left invariant under BBM. 
\end{remark}

We will prove below the following proposition. 
\begin{proposition}\label{pr.M32}
All the fixed points found in Theorem \ref{main-thm} a.s. belong to $\M_{3/2}$. 

Furthermore, the space $\M_{3/2}$ is left invariant under BBM in the following sense: for any point process $\theta\in \M_{3/2}$ a.s. and for any $t>0$, $\theta_t$ will also be a.s. in $\M_{3/2}$. 
\end{proposition}

Since $\M_{3/2}$ is preserved by our stochastic process and contains all the fixed points and since it is very large (i.e. it includes processes with super-exponential tail near $-\infty$), we believe it is a natural first step to analyse the basin of attraction of the fixed points from Theorem \ref{main-thm}. Another possible use is as follows: now that the space $\M_{3/2}$ avoids possible coming downs from $-\infty$, if one could further define a suitable metric $\rho$ on this space so that $(\M_{3/2},\rho)$ is a Polish space, we may then view the invariant measures as defined in Definition \ref{d.FP} in the more classical setting of Feller processes. 

The proof of Proposition \ref{pr.M32} is given in the next two subsections. 
\subsection{The fixed points belong to $\M_{3/2}$.}
To check that all fixed points belong a.s. to $\M_{3/2}$, it is enough to prove that $\widetilde{\calE}_\infty\in \M_{3/2}$ a.s. In fact we are going to show, based on \cite{cortines2019}, that a.s.
\begin{equation}\label{roughtailEt}
\widetilde{\calE}_\infty ([-x,\infty))\le x^3 e^{\sqrt{2}x} \textrm{ as } x \to +\infty,
\end{equation}
which suffices to get the a.s. finiteness of $\int_\R e^{-\beta |x|^{3/2}}\widetilde{\calE}_\infty(dx)$ for any $\beta>0$.

Recall \eqref{tilde-def},  one has $\widetilde{\calE}_\infty([-x,\infty))=\sum_{i}\ind{p_i\ge-x}\calD^{i}([-x-p_i,0])$.
As in Lemma 6.1 of \cite{cortines2019}, we have
\[
\widetilde{\calE}_\infty ([-x,\infty))=\sum_{i}\ind{p_i\in [-x,x]}\calD^{i}([-x-p_i,0])+\sum_{i}\ind{p_i>x}\calD^{i}([-x-p_i,0]).
\]
It is known from Proposition 1.5 of \cite{cortines2019} that there exists some constant $c_8>0$ such that for any $x\ge 0$, 
\[
\E[\calD[-x,0]]\leq c_8 e^{\sqrt{2}x}. 
\]
So,
\[
\E\left[\sum_{i}\ind{p_i\in [-x,x]}\calD^{i}([-x-p_i,0])\right]\leq \int_{-x}^{x} c_8\Cb_M e^{\sqrt{2}(x+y)} e^{-\sqrt{2} y}dy= 2c_8\Cb_M xe^{\sqrt{2} x}.
\]
Then, Markov inequality and Borel-Cantelli Lemma imply that
\[
\limsup_{x\to\infty}\frac{\sum_{i}\ind{p_i\in [-x,x]}\calD^{i}([-x-p_i,0])}{x^3e^{\sqrt{2}x}}=0, \textrm{ a.s. }
\]
On the other hand, $\sum_{i}\ind{p_i>x}\calD^{i}([-x-p_i,0])\ge1$ if and only if $\sum_{i}\ind{p_i>x}\ge1$. For a Poisson point process $\sum_i\delta_{p_i}$, it is clear that as $x\to\infty$, a.s.,
\[
\sum_{i}\ind{p_i>x}\to 0.
\]
So, $\sum_{i}\ind{p_i>x}\calD^{i}([-x-p_i,0])\to 0$ a.s. as $x\to\infty$. We hence conclude \eqref{roughtailEt}.

\subsection{The space $\M_{3/2}$ is invariant under BBM.}
Next, let us check that started from a point process $\theta_0$ with $\theta_0\in\M_{3/2}$ a.s., for any time $t>0$, we still have $\theta_t\in\M_{3/2}$ a.s. Recall that if $\theta_0=\sum_{i=I}\delta_{x_i}$, then $\theta_t$ can be written as
\[
\theta_t=\sum_{i\in I}\sum_{k=1}^{n^i(t)}\delta_{x_i+\chi_k^i(t)-\sqrt{2}t}
\]
with i.i.d. BBMs $(\{\chi^i_k(t);1\leq k\leq n(t)\}, t\ge0)$. Let $f_\beta(x):=e^{-\beta |x|^{3/2}}$, we are going to show that if $\theta_0\in \M_{3/2}$ a.s., then there exists some $\beta\in(0,1)$ such that
\begin{equation}
\int_{\R} f_\beta(x)\theta_t(dx)<\infty \textrm{ and } \esssup\theta_t<\infty, \textrm {a.s.}
\end{equation}
It is easy to see that $\esssup\theta_t<\infty$ a.s. if and only if $\theta_t(\R_+)<\infty$ a.s. So, we only need to check that 
\begin{equation}
\int_{\R} f_\beta(x)\theta_t(dx)<\infty \textrm{ and } \theta_t(\R_+)<\infty,\textrm{ a.s. }
\end{equation}
In fact, by considering the conditional expectation, it suffices to show that  if $\theta_0\in \M_{3/2}$ a.s., then there exists some $\beta\in(0,1)$ such that
\begin{equation}
\E\left[\int_{\R} f_\beta(x)\theta_t(dx)\Big\vert \theta_0\right]<\infty \textrm{ and } \E[\theta_t(\R_+)\vert \theta_0]<\infty,\textrm{ a.s. }
\end{equation}
We first consider $\theta_t(\R_+)=\sum_{i\in I}\sum_{k=1}^{n^i(t)}\ind{x_i+\chi_k^i(t)-\sqrt{2}t\ge0}$. Its conditional expectation is 
\begin{align}\label{bd-theta-R+}
\E[\theta_t(\R_+)\vert \theta_0]=&\int_\R \E\left[\sum_{k=1}^{n(t)}\ind{x+\chi_k(t)-\sqrt{2}t\ge0}\right]\theta_0(dx)\nonumber\\
\leq & \int_{-\infty}^0 \E\left[\sum_{k=1}^{n(t)}\ind{x+\chi_k(t)-\sqrt{2}t\ge0}\right]\theta_0(dx) +\int_{\R_+}\E[n(t)]\theta_0(dx),
\end{align}
where it is known that $\E[n(t)]=e^t$. So, as $\theta_0\in\M$, a.s.,
\[
\int_{\R_+}\E[n(t)]\theta_0(dx)= e^t \theta_0(\R_+)<\infty, 
\]
Let $(W_t, t\ge0)$ be a standard BM. By the well-known Many-to-One Lemma, one has
\begin{align*}
\E\left[\sum_{k=1}^{n(t)}\ind{x+\chi_k(t)-\sqrt{2}t\ge0}\right]=&\E[n(t)]\P(x+W_t-\sqrt{2}t\ge0)\\
=&e^t\int_{\sqrt{2}t -x}^{\infty} \frac{1}{\sqrt{2\pi t}}e^{-\frac{z^2}{2t}}dz
\end{align*}
which is less than $\frac{e^t}{2\sqrt{\pi t}}e^{-\frac{(\sqrt{2}t-x)^2}{2t}}$ for any $x<0$ and $t>0$. As a result, the first integral on the right hand side of \eqref{bd-theta-R+} is bounded by
\[
 \int_{-\infty}^0\frac{e^t}{2\sqrt{\pi t}}e^{-\frac{(\sqrt{2}t-x)^2}{2t}}\theta_0(dx)\leq  \int_{-\infty}^0 \frac{1}{2\sqrt{\pi t}}e^{-\sqrt{2}|x|-\frac{x^2}{2t}}\theta_0(dx).
\]
As $\theta_0\in \M_{3/2}$ a.s., there exists some $\beta_0=\beta_0(\theta_0)\in(0,1)$ such that 
\[
\int_{\R} e^{-\beta_0 |x|^{3/2}}\theta_0(dx)<\infty, \textrm{ a.s.}
\]
One sees that if $|x|\ge (2\beta_0t)^2$, $\frac{x^2}{2t}\ge \beta_0 |x|^{3/2}$. Therefore,
\begin{align*}
&\int_{-\infty}^0 \E\left[\sum_{k=1}^{n(t)}\ind{x+\chi_k(t)-\sqrt{2}t\ge0}\right]\theta_0(dx) \le \int_{-\infty}^0 \frac{1}{2\sqrt{\pi t}}e^{-\sqrt{2}|x|-\frac{x^2}{2t}}\theta_0(dx)\\
\leq & \frac{1}{2\sqrt{\pi t}}\int_{-\infty}^{-(2\beta_0t)^2}e^{-\beta_0 |x|^{3/2}}\theta_0(dx)+  \frac{1}{2\sqrt{\pi t}}\theta_0([-(2\beta_0t)^2,0])
\end{align*}
which is finite a.s. Going back to \eqref{bd-theta-R+}, one obtains that $\E[\theta_t(\R_+)\vert \theta_0]<\infty$ a.s. and thus deduces that $\theta_t(\R_+)<\infty$ a.s.

Similarly, for $\E\left[\int_{\R} f_\beta(x)\theta_t(dx)\Big\vert \theta_0\right]$, by the Many-to-One Lemma, one has
\begin{align*}
\E\left[\int_{\R} f_\beta(x)\theta_t(dx)\Big\vert \theta_0\right]=& \int_\R \E\left[\sum_{k=1}^{n(t)}f_\beta(x+\chi_k(t)-\sqrt{2}t)\right]\theta_0(dx)\\
\leq & \int_{-\infty}^{-A_t} e^t \E[e^{-\beta |x+W_t-\sqrt{2}t|^{3/2}}] \theta_0(dx)+ e^t \theta_0([-A_t,+\infty))
\end{align*}
for any $A_t>0$. Apparently, $e^t \theta_0([-A_t,+\infty))<\infty$ a.s. It remains to study $\int_{-\infty}^{-A_t} e^t \E[e^{-\beta |x+W_t-\sqrt{2}t|^{3/2}}] \theta_0(dx)$. In fact, 
\begin{align*}
\int_{-\infty}^{-A_t} e^t \E[e^{-\beta |x+W_t-\sqrt{2}t|^{3/2}}] \theta_0(dx)& =\int_{-\infty}^{-A_t} e^t \int_\R e^{-\beta |x+z-\sqrt{2}t|^{3/2}} \frac{1}{\sqrt{2\pi t}}e^{-\frac{z^2}{2t}}dz \theta_0(dx)\\
& =\int_{-\infty}^{-A_t}\frac{1}{\sqrt{2\pi t}}\int_\R e^{-\beta |x+z|^{3/2}-\sqrt{2}z-\frac{z^2}{2t}}dz\theta_0(dx)\,.
\end{align*}
Now by a tedious but standard Laplace method, we claim that one can choose $A_t$ sufficiently large so that for all $|x|\ge A_t$ one has 
\[
\frac{1}{\sqrt{2\pi t}}\int_\R e^{-\beta |x+z|^{3/2}-\sqrt{2}z-\frac{z^2}{2t}}dz\le C_t e^{-\beta |x|^{3/2}+K_t |x|}
\]
with some $C_t>0$ and $K_t>0$. We may now take $\beta\in(\beta_0,1)$ so that $\beta_0|x|^{3/2}\le \beta |x|^{3/2}-K_t |x|$ for $|x|$ sufficiently large. We then end up with
\[
 \int_{-\infty}^{-A_t} e^t \E[e^{-\beta |x+W_t-\sqrt{2}t|^{3/2}}] \theta_0(dx)\leq  \int_{-\infty}^{-A_t} C_t e^{-\beta_0|x|^{3/2}}\theta_0(dx)<\infty, \textrm{ a.s.}
\]
This completes the proof.  \qed

\section{Concluding remarks and questions}\label{s.conclusion}

\subsection{Extending Kabluchko's analysis of super-critical drift.}\label{non-critical-drift}
We have considered through this paper the BBM resulting from the Brownian motion with the drift $\pm\sqrt{2}$. If we study the BBM with drift $\lambda $ for some $|\lambda |> \sqrt{2}$, the problem of characterizing the corresponding fixed points was addressed by Kabluchko in \cite{kabluchko}. The author showed the existence and uniqueness of fixed points with an assumption of locally finite intensity measure. Interestingly, as in \cite{liggett78}, his proof is also based on the Choquet-Deny convolution equation \cite{ChoquetDeny}, \cite{Deny}, though its use is quite different. Indeed Choquet-Deny equation is shown to be satisfied by the intensity measure of any fixed point. 

However, let us point out that for BBM with critical drift, the intensity measure $\E[\theta(\cdot)]$ of any (non-zero) fixed point $\theta$ is infinite on any interval. In particular Kabluchko's argument fails to work at criticality. Our new idea in this paper is tailor-made to circumvent this issue. Moreover, we believe our method may 
also be generalized to the case with super-critical drift $\lambda > \sqrt{2}$ and may as such provide a new proof of Kabluchko's theorem without the assumption of locally finite intensity measure. 

\subsection{A short inspection of the lattice case for BRWs.}\label{ss.BRW}

It is natural to consider the same question for branching random walks (BRW) in discrete time, i.e. to find the fixed points of the particles system obtained by attaching to each atom of some point process an independent BRW with critical drift. Usually, we take a BRW $(\{S_u; |u|=n\}, n\ge0)$ in the boundary case (see e.g. \cite{BK05,shi-book}) where $M_n$, the minimum of BRW has zero velocity. It has been proved by A\"id\'ekon \cite{Aid13} that if the walks are not lattice, then under some mild condition, $M_n-\frac32 \log n$ converges in law to some random variable. Later, Madaule \cite{Mad16} showed that $\sum_{|u|=n}\delta_{S_u-\frac32\log n}$ converges in law to some limiting extremal point process which also has the decoration structure as in BBM case. Thanks to these properties, we believe that our techniques may apply to show in this non-lattice setting that any fixed point should correspond to the above extremal point process with some shift.

More interestingly, the question of existence/uniqueness of fixed point also makes perfect sense in the lattice case. Let us say a few words on this intriguing situation.
For a BRW $(\{S_u; |u|=n\}, n\ge0)$ in the lattice case (say in $\Z$) and in the boundary case, we still have the tightness of $M_n-\frac32\log n$, but the weak convergence fails. However, we expect that we should still have the convergence in law of $\sum_{|u|=n}\delta_{S_u-M_n}$ conditioned on $M_n$ being extremely small/large. If so, this means that the limiting decoration process can be defined. To proceed further, we would actually need the joint convergence of $M_n+a_n$ and $\sum_{|u|=n}\delta_{S_u-M_n}$ conditioned on $M_n\le -a_n$ with $a_n\in\mathbb{N}$ of order $\sqrt{n}$. Following the ideas of \cite{ABK13}, it is natural to start from a Poisson point process on $\N$ with a suitable intensity  $\mu$ so that $\sum_{k\in\N}\P(M_n \le -k)\mu(k)\approx \sum_{k=\delta\sqrt{n}}^{\sqrt{n}/\delta}\P(M_n\le -k)\mu(k)$ converges to some positive constant. Then as time goes on, the particle system should stabilize.  One would end up this way with some limiting point process which may serve as an equilibrium measure for this particle system.  

In this way, we expect the following structure for the fixed points in the lattice case: we first sample a Poisson point process with intensity $\sum_{k\in \Z} C p(k)\delta_k$ where
\[
p(k)=\lim_{n\to\infty}\frac{\P(M_n+a_n=-k)}{\P(M_n+a_n\le 0)}  \text{ (where recall that $a_n$ is of order $\sqrt{n}$)} 
\]
comes from the conditioned law of $M_n+a_n$ given $M_n\le -a_n$. Here $C>0$
could be a positive constant or a positive random variable. To each atom, we then attach an independent decoration process which is the limit in law of $\sum_{|u|=n}\delta_{S_u-M_n}$ conditioned on $M_n\le -a_n$.

Due to this expected structure, we feel that this point of view of {\em fixed points} is particularly well adapted to the study of BRWs in the lattice case.

Let us summarize the above discussion into the following open question.
\begin{question}
Show both for non-lattice and  lattice BRWs that fixed point exists and are all given by Poisson decorated point processes. 
\end{question}

\begin{figure}[!htp]
\begin{center}
\includegraphics[width=0.9\textwidth]{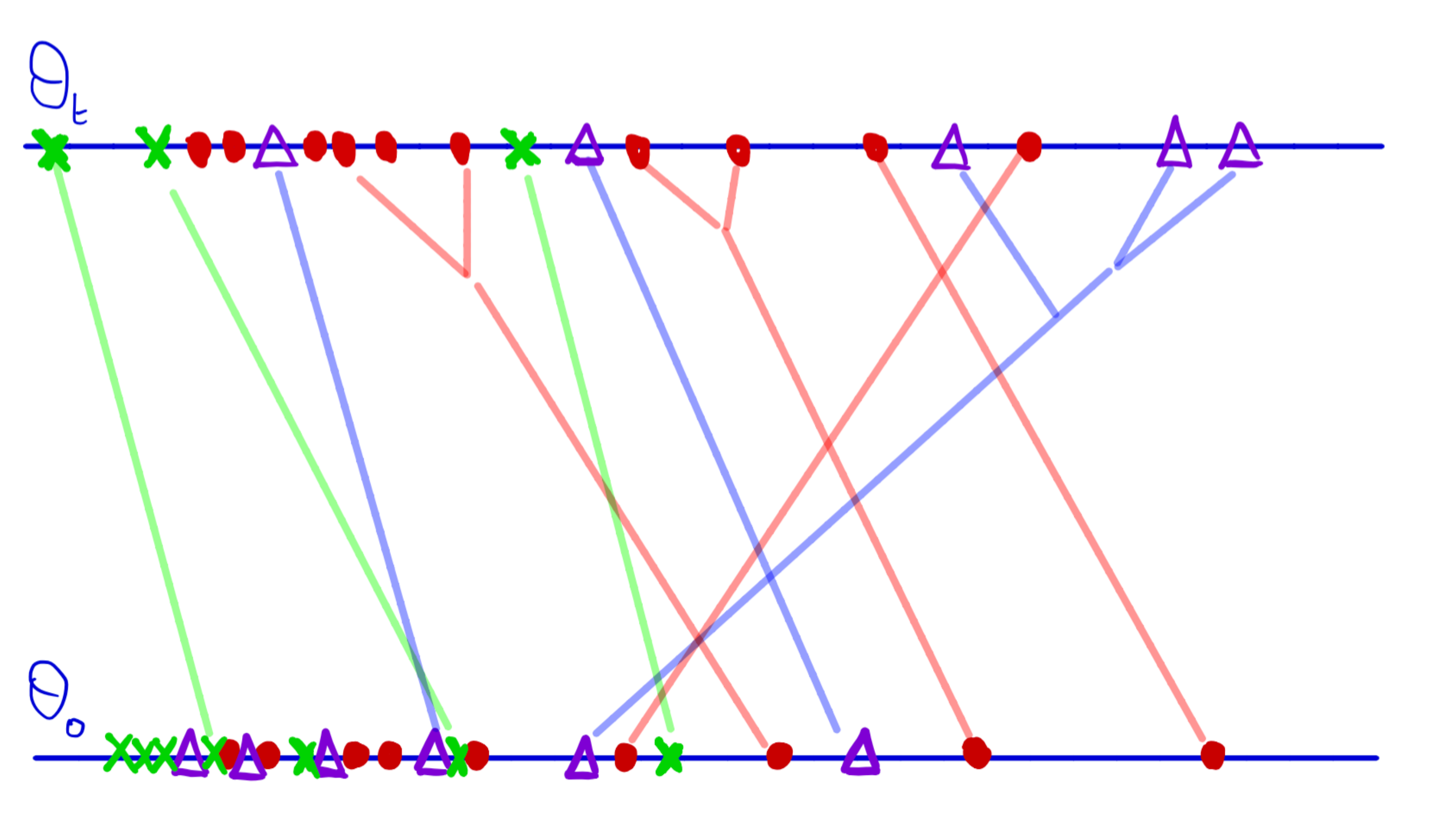}
\end{center}
\caption{}\label{f.shuffle}
\end{figure}
\subsection{Fixed points modulo translations.}\label{ss.aizenman}

In the works \cite{AR05,AA09}, the authors analyze the so-called {\em quasi-stationary} point processes which generalize Liggett's invariant point processes from \cite{liggett78}. To define what they are, let us denote a point process $\theta\in \M$ (see Section \ref{ss.state}) as $\theta= \sum_{i=1}^\infty \delta_{x_i}$ (we only treat the countable case here as there are no finite quasi-stationary states besides $0$). A {\em quasi-stationary} state is a point process $\theta$ for which the joint law of the gaps 
\[
\{X_i-X_{i+1}\}_{i\geq 1}\,,
\] 
remains invariant under the considered dynamics. Equivalently {\em quasi-stationary} states are invariant point processes viewed modulo translations. In \cite{AR05,AA09} such {\em quasi-stationaty} states are characterized for Liggett's non-interacting diffusions as well as generalizations where correlations are introduced between particles depending on their respective ranks. In the non-interacting case, it is shown in \cite{AR05} that the only fixed points are again given by superpositions of Liggett's fixed points. 

In our present setting, it is therefore a natural question to ask what are the {\em quasi-invariant} states of branching Brownian motion. Note that we do not need to specify a drift in this case as we view point processes modulo translations and this has the effect to quotient out global drifts. By definition any {\em invariant} point process must be {\em quasi-invariant} but the reverse may not be true.  We conjecture though that for BBM, the situation is as for Liggett's non-interacting diffusions \cite{AR05}. Namely,
 
\begin{question}
Show that all point processes $\theta$ viewed modulo translations and which are invariant under BBM (no need to precise a drift here as we view the point process modulo translations) are given by Kabluchko's super-critical fixed points together with our critical fixed points (and nothing else). 
\end{question}

Let us make an important comment here. There is a greater variety of possible candidates for other types of {\em quasi-invariant} states in the case of BBM: the main one being the decoration processes $\bar \calE_\infty$ or $\calE_\infty$ (when viewed modulo translations they lead to the same gap process). One may believe at first sight that the decoration process could lead to a {\em quasi-invariant} state. Indeed if one looks at Figure \ref{f.shuffle}, the law of the point process made of red dots and triangular dots at times $0$ and $t$ is the same (starting from a fixed point $\theta_0$). More precisely the joint law of the first two leaders together with their decoration processes are the same at times 0 and time $t$. In the same fashion, the joint law of triangles and crosses at time 0 is the same as the joint law of red dots and crosses at time $t$ (again in the sense that second and third leader together with their decoration are invariant in law). One may readily conclude from this set of equalities in law that  a leader together with its decoration ancestors should lead to a {\em quasi-invariant} state. (Say, on the figure the red dots viewed modulo translations). It is not hard to show that it is in fact not true. Note that this does not contradict the above identities as one should also take into account the effect of the random permutation $\sigma_t$ from time $0$ to time $t$ which reshuffles the order between leaders. The fact $\calE_\infty$ is now excluded from the possible {\em quasi-invariant} point processes gives more support to the above open question. 

\subsection{Remaining questions.}
We end this paper with some remaining open questions.

\begin{question}\label{q.calN}
Prove our main Theorem \ref{main-thm} without our assumption that point processes $\theta$ must have a top particle. I.e. characterize all fixed-points supported in $\calN$ rather than in $\M\subset \calN$. 
\end{question}


\begin{question}
Identify the basin of attraction of the fixed points of BBM with critical drift, say inside the space $\M_{3/2}$ introduced in Definition \ref{d.M32}. 
\end{question}

\begin{question}
In the context of random matrices and determinantal processes, show that the only fixed points to the Markov processes introduced by Najnudel and Virag in \cite{najnudel} are given by  the $\mathrm{Sine}_\beta$ point processes. 
\end{question}

\bibliographystyle{alpha}
\bibliographystyle{acm}	
\bibliography{biblio}

\end{document}